\documentclass[final,leqno,letterpaper]{etna}
\usepackage{amsmath,amssymb,amsfonts,stmaryrd}
\usepackage{soul}
\usepackage{array}
\usepackage{empheq}
\usepackage{xfrac}
\usepackage{minibox}
\usepackage{enumerate}
\usepackage{color}
\usepackage{subcaption, booktabs}
\usepackage[space,noadjust,nocompress]{cite}

\usepackage{lineno}

\allowdisplaybreaks

\newcommand{\vect}[1]{\mathbf{#1}}

\DeclareMathOperator{\Ima}{Range}
\DeclareMathOperator{\tr}{tr}
\DeclareMathOperator{\argmin}{argmin\text{ }}

\setbibdata{1}{xx}{46}{2017} 
\shorttitle{ENERGY MINIMIZATION IN AMG INTERPOLATION}
\shortauthor{J. BRANNICK, S. MACLACHLAN, J. SCHRODER, AND B. S. SOUTHWORTH}
\title{The role of energy minimization in algebraic multigrid interpolation\thanks{Received... Accepted... Published online on... Recommended by....
The work of Ben Southworth was supported by the Department of Defense (DoD)
  through the National Defense Science \& Engineering Graduate
  Fellowship (NDSEG) Program.  The work of Scott MacLachlan was
  partially supported by an NSERC Discovery Grant, and by the NSF
  under award DMS-0811022.  The work of Jacob Schroder was partially
  supported by the Department of Energy/NNSA ASC program. 
  }}

\author{  James Brannick
  \footnotemark[2]
  \and
  Scott P. MacLachlan
  \footnotemark[3]
  \and
  Jacob B. Schroder
  \footnotemark[4]
  \and
  Ben S. Southworth
  \footnotemark[5]
}

\ifpdf\hypersetup{  pdftitle={The Role of Energy Minimization in Algebraic Multigrid Interpolation},
  pdfauthor={Brannick, MacLachlan, Schroder, and Southworth}
  pdfkeywords={Algebraic Multigrid, Trace Minimization}
}
\fi

\begin{document}

\maketitle \renewcommand{\thefootnote}{\fnsymbol{footnote}}
\footnotetext[2]{Department of Mathematics, Penn State ({\tt
    brannick@psu.edu})} \footnotetext[3]{Department of Mathematics and
  Statistics, Memorial University of Newfoundland ({\tt
    smaclachlan@mun.ca})} \footnotetext[4]{Department of Mathematics
  and Statistics, University of New Mexico ({\tt jbschroder@unm.edu})
} \footnotetext[5]{Department of Applied Mathematics, University of
  Colorado at Boulder ({\tt ben.s.southworth@gmail.com})}

\begin{abstract}
Algebraic multigrid (AMG) methods are powerful solvers with linear or near-linear computational complexity
for certain classes of linear systems, $A\vect{x}=\vect{b}$. Broadening the scope of problems that AMG can effectively solve
requires the development of improved interpolation operators. Such development is often based on AMG convergence
theory. However, convergence theory in AMG tends to have a disconnect with AMG in practice due to the practical constraints
of (i) maintaining matrix sparsity in transfer and coarse-grid operators, and (ii) retaining linear complexity in the setup and
solve phase. This paper presents a review of fundamental results in AMG convergence theory, followed by
a discussion on how these results can be used to motivate interpolation operators in practice. A general
weighted energy minimization functional is then proposed to form interpolation operators, and a novel ``diagonal''
preconditioner for Sylvester- or Lyapunov-type equations developed simultaneously. 
Although results based on the weighted energy minimization typically underperform compared to a fully constrained
energy minimization, numerical results provide new insight into the role of energy minimization and
constraint vectors in AMG interpolation. 
\end{abstract}
\begin{keywords}
  algebraic multigrid, trace minimization
\end{keywords}
\begin{AMS}
AMS subject classifications
\end{AMS}

\section{Introduction}

Algebraic multigrid (AMG) was designed as a solver for large, sparse linear systems,
typically M-matrices, resulting from the discretization of elliptic PDEs. For many such
problems, AMG has been shown to achieve fast convergence, and scale in parallel to hundreds of thousands of
processors \cite{Baker:2012ko}. Such convergence and scaling properties are desirable
for solvers, and substantial work has been devoted to broadening the applicability of AMG.
Continued research has made for a rich theoretical basis for AMG \cite{Falgout:2004cs,
Falgout:2005hm, Vassilevski:2008wd, Vassilevski:2010vy, MacLachlan:2014di},
as well as many numerical implementations and variations that are either
robust for a larger class of linear systems \cite{Brezina:2004eh,DAmbra:2013iwa,
Manteuffel:2016vd}, or effective at solving a specific problem such as linear elasticity
\cite{Baker:2009va} or Hemholtz \cite{Olson:2010bh}.  Nevertheless, the one-size-fits-all AMG
solver remains elusive, in part because many of the theoretical results are
difficult to use in a practical setting. 

A novel feature of AMG in contrast to many other linear solvers is
that the setup and solve complexity in terms of floating point
operations (FLOPs) are both typically linear or log-linear in the total number of degrees of freedom (DOFs).
This is fundamental to good scaling of time to solution with increasing
problem size, but also limits the options in algorithm design, particularly when trying
to directly use theoretical results on convergence. Two common aspects
seen in AMG convergence theory are the use of orthogonal projections
onto subspaces and requiring
a given approximation property to hold for all vectors. In both cases,
namely constructing an orthogonal projection or enforcing a constraint for $n$
basis vectors, the complexity of explicitly enforcing such requirements is at least quadratic in $n$
and, thus, not feasible in keeping with the desired linear complexity of AMG. Furthermore, 
AMG can only have linear complexity when all operators are sparse, including the coarse-grid
operators constructed for a multilevel algorithm. In the abstract setting of convergence
theory, such sparsity constraints are not accounted for, adding an additional
barrier to the direct use of convergence theory in practical methods. 

Here, we review the tension between theory and practice in AMG and propose
a new variant of AMG that aims to directly address these complications. An overview of AMG convergence theory is given in Section \ref{sec:theory}.
Fundamental results on two-grid and multigrid convergence theory are presented in a simple and
consistent manner, to clarify what is required of interpolation operators for effective
AMG convergence, and the so-called ``optimal'' and ``ideal'' interpolation operators
are introduced. Section \ref{sec:practice} proposes a discussion of AMG
interpolation operators used in practice, and how they relate to theoretical results,
along with an examination of how different theoretical results can be approximated
in linear complexity. This leads to the introduction of a general weighted functional
to be minimized in forming interpolation operators in Section \ref{sec:functional}, which is
shown to have a unique solution for a fixed interpolation sparsity pattern. 
A conjugate gradient method is developed to approximate the solution, with a novel
preconditioner that is applicable to general equations with a Sylvester- or Lyapunov-like form
(Section \ref{sec:pcg}). Numerical results demonstrate that a constrained
energy minimization \cite{Brannick:2007fb,Mandel:1999wg,Olson:2011fg,Wan:1999ky}
consistently outperforms a weighted energy minimization. Although this may seem intuitive for
AMG researchers, a number of other interesting results also come up that lead to
open questions on interpolation in AMG:
\begin{itemize}
\item Enforcing one constraint vector to be (almost) exactly in the range of interpolation is fundamental to good AMG convergence.
However, adding additional constraint vectors that are not effectively reduced by the current AMG hierarchy does not necessarily
improve convergence, which is at odds with motivation of traditional adaptive AMG methods.

\item Energy-minimization applied to columns of $P$ (while maintaining constraints) is also fundamental to a
convergent AMG method for some more difficult model problems. However, although further iterations of energy minimization
continue to reduce the associated residual, convergence of the
resulting AMG solver does not improve after a small
number of iterations. 

\item Using a diagonal preconditioner for energy-minimization iterations applied to columns of $P$ can offer significant improvement
in convergence of the resulting AMG solver.
\end{itemize}

\section{Theoretical framework}\label{sec:theory}

Multilevel solvers come in various forms, including geometric multigrid (GMG),
AMG, finite element algebraic multigrid (AMGe), algebraic multilevel iterations (AMLI),
and the method of subspace corrections. In this work, we focus on AMG as
a general method to solve a linear system $A\vect{x} = \vect{b}$ using
only ``algebraic'' information, in contrast to GMG and
AMGe, which require additional information, on the underlying grid or finite element stiffness matrices,
respectively. Subspace corrections are presented in a more general yet framework than AMG,
but analysis of subspace correction can also be applied to AMG \cite{Vassilevski:2008wd}.

The basis for AMG as an iterative method to solve $A\vect{x}=\vect{b}$ is in reducing error
through two processes: ``relaxation'' and ``coarse-grid correction.'' If designed properly,
these processes are complementary in the sense that they are effective on
different error modes and, together, effectively reduce all types of error.
Relaxation refers to a general iterative method of the form
\[
\mathbf{x}_{k+1}  = \mathbf{x}_k + M^{-1}(\mathbf{b} - A\mathbf{x}_k),
\]
and is often chosen to be a simple method such as Jacobi or Gauss-Seidel.
This process is typically efficient at removing ``high-frequency error,'' or error
associated with large eigenvalues of $A$. Convergence of a relaxation scheme
in the ``energy norm'' or ``$A$-norm,'' $\|\mathbf{v}\|_A^2 = \langle A\mathbf{x},\mathbf{x}\rangle$, is equivalent
to bounding the error-propagation matrix in the $A$-norm, namely
$\|I - M^{-1}A\|_A < 1$. Furthermore, $\|I - M^{-1}A\|_A < 1$ if and only if
$M+M^T-A$ is symmetric positive definite (SPD) \cite[Theorem 2.3.1]{Vassilevski:2010vy}.  In this
case, we say that $M$ is an $A$-convergent relaxation operator.

Multigrid originated in the geometric setting, where high-frequency error
actually has a high physical frequency. Since standard relaxation schemes
such as Jacobi and Gauss-Seidel are able to capture this error well, the natural
way to capture the converse, low-frequency error, is to recursively coarsen
the underlying grid so that low-frequency modes on the initial fine grid
appear high-frequency on coarser grids. Relaxation on coarser grids will
then reduce this error, and the results can be interpolated back to the fine grid.
The algebraic concept is much the same but low-frequency refers to algebraically smooth
modes, corresponding to large eigenvalues of $I - M^{-1}A$ or, typically, small eigenvalues of $A$,
and vice-versa for high-frequency. Because there is no explicit grid in the
algebraic setting, algebraic coarsening is based on choosing a coarse subspace
which can capture algebraically smooth error from the fine grid.

For an $n\times n$ SPD matrix $A$, consider an $\ell^2$-orthogonal decomposition of
$\mathbb{R}^n$, where any $\mathbf{x}$ can be decomposed
as $\vect{x} = R^T\vect{y}_c + S\vect{y}_f$, where $RS = 0$. Here, $S$ corresponds
to the space on which relaxation is effective, and $R$ defines the coarse space on which a coarse-grid correction
is constructed. Defining the interpolation operator, $P$, we assume that $PR$ is a projection, which requires $RP = I$.
A Galerkin coarse-grid operator is formed, $A_c = P^TAP$, and an exact coarse-grid correction (in the $A$-norm) given
by the $A$-orthogonal projection onto $\Ima(P)$, $\pi_A = P(P^TAP)^{-1}P^TA$.
In this work, a CF-style splitting will be used (or in the case of
aggregation-based coarsening, a root-node approach, where one node in each
aggregate is declared a C-point and the rest F-points \cite{Manteuffel:2016vd}).
A CF-splitting has the useful property that coarse-grid nodes are a subset of
nodes on the current grid, allowing for $A$ to be written
in the block form
\[
A = \begin{bmatrix} A_{ff} & A_{fc} \\ A_{cf} & A_{cc} \end{bmatrix}.
\]
In this case, splitting operators take the form $R = \begin{bmatrix} 0 & I\end{bmatrix}$,
$S = \begin{bmatrix} I & 0\end{bmatrix}^T$, and $P = \begin{bmatrix} W^T & I\end{bmatrix}^T$,
where $RAR^T = A_{cc}$ and $S^TAS = A_{ff}$. 
Together, the two-grid error-propagation matrix operator for AMG, with an $A$-symmetric relaxation scheme
based on $M^{-1}$ and $M^{-T}$, is given by
\[
E_{TG} = (I - M^{-T}A)(I - \pi_A)(I - M^{-1}A).
\]

Convergence of $E_{TG}$ is generally considered in the $A$-norm, where each iteration reduces
error in the $A$-norm by at least a factor of $\|E_{TG}\|_A$.  Noting that $E_{TG}$ is symmetric in the
$A$-norm, it follows that eigenvectors of $E_{TG}$ are $A$-orthogonal and $\|E_{TG}\|_A = \rho(E_{TG})$.  
Thus, optimizing the AMG convergence rate can be viewed
equivalently as minimizing $\|E_{TG}\|_A$ or $\rho(E_{TG})$. Two-grid convergence
can also be considered in terms of the spectral equivalence between a preconditioner, $B_{TG}$, and $A$,
where $E_{TG} = I - B_{TG}^{-1}A$ \cite[Proposition 5.1.2]{Vassilevski:2010vy}; however, here we 
bound convergence in terms of $\|E_{TG}\|_A$ for consistency. A multilevel method is implemented
and analyzed as a two-grid method with an inexact coarse-grid solve, where the coarse-grid ``solve'' recursively
calls a two-grid method on the coarse-grid problem. Multigrid convergence is also considered in the $A$-norm,
where we want to bound $\|E_{MG}\|_A \leq K < 1$.

Let $M$ and $M^T$ be $A$-convergent relaxation operators and define the symmetrized relaxation operator
as $\widetilde{M} = M^T(M+M^T-A)^{-1}M$,  so that $I-\widetilde{M}^{-1}A = (I-M^{-1}A)(I - M^{-T}A)$. 
This symmetrizes the action of $M$ and is used primarily as a
theoretical tool (as $\widetilde{M}$ is rarely easily
computable). Common bounds on two-grid convergence factors come from considering various orthogonal
projections onto the range of the interpolation operator, $P$.  Define
$\pi_X := P(P^TXP)^{-1}P^TX$ as the unique $X$-orthogonal projection
onto $\Ima(P)$ for some nonsingular operator $X$, e.g. $A$ or $\widetilde{M}$, and
$Q_P := P(P^TP)^{-1}P$ as the $l^2$-orthogonal projection onto $\Ima(P)$. 
For computable bounds, assume that $X$ is spectrally equivalent to $\widetilde{M}$, denoted $X\simeq \widetilde{M}$;
that is, there exists $0<c_1\leq c_2$  such that
\begin{equation}
c_1\mathbf{v}^TX\mathbf{v} \leq \mathbf{v}^T\widetilde{M}\mathbf{v} \leq c_2\mathbf{v}^TX\mathbf{v}. \label{eq:spec_equiv}
\end{equation}
\begin{linenomath}
It follows from \eqref{eq:spec_equiv} and the definition of orthogonal projections that
\begin{align}
\|(I-\pi_X)\mathbf{v}\|_X^2 & \leq \|(I-\pi_{\widetilde{M}})\mathbf{v}\|_X^2 \leq \frac{1}{c_1}\|(I-\pi_{\widetilde{M}})\mathbf{v}\|_{\widetilde{M}}^2 \label{eq:X_M1} \\
\|(I-\pi_{\widetilde{M}})\mathbf{v}\|_{\widetilde{M}}^2 & \leq \|(I-\pi_X)\mathbf{v}\|_{\widetilde{M}}^2 \leq c_2\|(I-\pi_X)\mathbf{v}\|_X^2 \label{eq:X_M2}
\end{align}
for all $\mathbf{v}$.
\end{linenomath}

\begin{linenomath}
Finally, note the following identities with respect to the Frobenius inner product and trace that are used regularly in this work:
\begin{align*}
\langle A,B\rangle_F & = \sum_{ij}A_{ij}B_{ij} = \tr(B^TA) = \tr(A^TB), \\
\tr(ABC) & = \tr(CAB) = \tr(BCA),
\end{align*}
and let $A\circ B$ denote the Hadamard product, defined as the element-wise multiplication of two matrices
$A,B\in\mathbb{R}^{m\times n}$.
\end{linenomath}

\subsection{Two-grid convergence}

Substantial work has been devoted to understanding convergence theory of AMG
in the two-grid setting \cite{Lee:2008iu,Notay:2014uc,Vassilevski:2010vy,
Vassilevski:2008wd,Falgout:2004cs,Falgout:2005hm,MacLachlan:2014di,Zikatanov:2008jp}.
By Lemma 4.1 of \cite{SFMcCormick_1984b}, we can analyze
$\|E_{TG}\|_A$ either directly, or by considering the variants with only
pre- or post-relaxation, as
\[
\|E_{TG}\|_A = \|(I-\pi_A)(I-M^{-1}A)\|_A^2 = \|(I-M^{-T}A)(I-\pi_A)\|_A^2.
\]
One of the simplest two-grid convergence bounds is given by Lemma 2.3
of \cite{SFMcCormick_1985b}:
\begin{theorem}
  If there is a $\delta > 0$ such that
  \[
  \|(I-M^{-T}A)\vect{v}\|_A^2 \leq \|\vect{v}\|_A^2 - \delta
  \|(I-\pi_A)\vect{v}\|_A^2,
  \]
  for all $\vect{v}$, then $\|E_{TG}\|_A \leq 1-\delta$.
\end{theorem}

Following \cite{JWRuge_KStuben_1987a, MacLachlan:2014di}, sufficient
conditions for two-grid convergence are given in the following theorem
\begin{theorem}
  \label{thm:classical}
Let symmetric and positive-definite matrix $X$ be given, and assume
that there exist $\alpha,\beta > 0$ such that
$\|(I-M^{-T}A)\vect{v}\|_A^2 \leq \|\vect{v}\|_A^2 -
\alpha\|A\vect{v}\|_X^2$ and $\|(I-\pi_A)\vect{v}\|_A^2 \leq \beta
\|A\vect{v}\|_X^2$ for all $\vect{v}$.  Then, $\|E_{TG}\|_A^2 \leq 1-\alpha/\beta$.
\end{theorem}
\begin{proof}
  \begin{linenomath}
We prove the bound on $\|E_{TG}\|_A$ by proving the corresponding bound
on $\|(I-M^{-T}A)(I-\pi_A)\|_A$.  For any $\vect{v}$,
\begin{align*}
\|(I-M^{-T}A)(I-\pi_A)\vect{v}\|_A^2 & \leq \|(I-\pi_A)\vect{v}\|_A^2
- \alpha\|A(I-\pi_A)\vect{v}\|_X^2 \\
& \leq \|(I-\pi_A)\vect{v}\|_A^2 -
\alpha/\beta\|(I-\pi_A)^2\vect{v}\|_A^2 \\
& \leq (1-\alpha/\beta)\|\vect{v}\|_A^2.
\end{align*}
\end{linenomath}
\end{proof}

The first assumption on Theorem \ref{thm:classical} is commonly
referred to as the smoothing property, since it assumes that
relaxation effectively reduces the error in an approximation when the
residual associated with that error is large (when measured in the
$X$-norm).  The second assumption on Theorem \ref{thm:classical} is
referred to as the strong approximation property, since it
assumes that coarse-grid correction is effective at reducing errors
when the associated residuals are small; this is equivalent to
assuming that such errors are well-approximated within $\Ima(P)$.
This assumption is termed the {\it strong} approximation property as
it can clearly be replaced by a weaker one, that
$\|(I-\pi_A)\vect{v}\|_A^2 \leq \beta\|A(I-\pi_A)\vect{v}\|_X^2$ for
all vectors, $\vect{v}$, stating that coarse-grid correction is
effective at reducing errors for which the residual after coarse-grid
correction is small.  This latter assumption is commonly referred to
as the {\it weak} approximation property. The difference between the
weak and strong approximation properties comes up in the multilevel
setting, and is discussed in Section \ref{sec:theory:multi}.

In practice, the weak and strong approximation properties are
typically considered in slightly altered forms.  For the strong
approximation property, an equivalent statement is that for any
$\vect{v}$, there exists a $\vect{v}_c$ such that
\[
\|\vect{v}-P\vect{v}_c\|_A^2 \leq \beta \|A\vect{v}\|_X^2.
\]
While a similar equivalence could be derived for the weak
approximation property, a more typical bound arises by noting that,
for any $\vect{v}$ and $\vect{v}_c$,
\[
\|(I-\pi_A)\vect{v}\|_A^2 = \langle
A(I-\pi_A)\vect{v},(I-\pi_A)\vect{v}-P\vect{v}_c\rangle \leq \|A(I-\pi_A)\vect{v}\|_X\|(I-\pi_A)\vect{v}-P\vect{v}_c\|_{X^{-1}}.
\]
Thus, a sufficient condition for the weak approximation property to
hold is that for any $\vect{v}$ there exists a $\vect{v}_c$ such that
\begin{equation}
  \label{eq:wap_bestkind}
\|\vect{v}-P\vect{v}_c\|_{X^{-1}}^2 \leq \beta\|\vect{v}\|_A^2.
\end{equation}
Note that this is trivially true for $\vect{v} \in \Ima(P)$, but the
weak approximation property is implied by this condition for $\vect{v}
\in \Ima(P)^\perp$.  It is in this form that much of the recent
two-grid AMG theory has been developed.

Of particular interest is the result obtained taking $X =
\widetilde{M}^{-1}$, so that the smoothing property in Theorem
\ref{thm:classical} trivially holds with $\alpha = 1$.  In this case,
the two-grid convergence bound in Theorem \ref{thm:classical} is
determined entirely by the constant in the weak approximation
property.  Indeed, in this setting a sharp bound on convergence is
possible  \cite{Falgout:2005hm,Vassilevski:2010vy}.
\begin{theorem}[Weak approximation property]\label{th:wap}
Let $A$ be SPD, $\widetilde{M}= M^T(M+M^T-A)^{-1}M$ for some relaxation
scheme $M$, and $P$ the interpolation operator for a two-grid method. Suppose $\exists$ $K$ such that
for any $\mathbf{v}\neq 0$, there exists a $\mathbf{v}_c$ such that
\[
\frac{\|\mathbf{v} - P\mathbf{v}_c\|_{\widetilde{M}}^2}{\|\mathbf{v}\|_A^2}  \leq K.
\]
Then the two-grid method converges uniformly, and $\|E_{TG}\|_A \leq 1- \frac{1}{K}$.
Furthermore, the best (minimal) constant $K$ over all $P$ is given by
\begin{equation}
K_{TG} = \max_{\mathbf{v}\neq\mathbf{0}} \frac{\| ( I - \pi_{\widetilde{M}})\mathbf{v}\|^2_{\widetilde{M}}}{\|\mathbf{v}\|^2_A}, \label{eq:wap1}
\end{equation}
in which case $\|E_{TG}\|_A = 1- \frac{1}{K_{TG}}$.
\end{theorem}

\begin{linenomath}
Equation \eqref{eq:wap1} gives a sharp bound on two-grid convergence,
but can be generalized to any matrix $X$ and corresponding
$X$-orthogonal projection onto $\Ima(P), \pi_X$. Let $X$ be
spectrally equivalent to $\widetilde{M}$ as in equation \eqref{eq:spec_equiv}.
Then, from equations \eqref{eq:X_M1} and \eqref{eq:X_M2}, 
\begin{align}
c_1\max_{\mathbf{v}\neq0}\frac{\| ( I -
  \pi_X)\mathbf{v}\|^2_X}{\|\mathbf{v}\|^2_A} & \leq 
c_1\max_{\mathbf{v}\neq0}\frac{\| ( I -
  \pi_{\widetilde{M}})\mathbf{v}\|^2_X}{\|\mathbf{v}\|^2_A} \notag \\
& \leq
 K_{TG}  \leq
\max_{\mathbf{v}\neq0} \frac{\| ( I - \pi_X)\mathbf{v}\|^2_{\widetilde{M}}}{\|\mathbf{v}\|^2_A} \leq
c_2\max_{\mathbf{v}\neq0} \frac{\| ( I - \pi_X)\mathbf{v}\|^2_X}{\|\mathbf{v}\|^2_A} .\label{eq:wap_bounds}
\end{align}
In the case of $X=I$, \eqref{eq:wap_bounds} simplifies to considering
interpolation error in the $l^2$-norm \cite{Zikatanov:2008jp}
\[
\lambda_{\textnormal{min}}(\widetilde{M})\max_{\mathbf{v}\neq0}\frac{\| ( I - Q_P)\mathbf{v}\|^2}{\|\mathbf{v}\|^2_A} \leq K_{TG} \leq 
	\lambda_{\textnormal{max}}(\widetilde{M}) \max_{\mathbf{v}\neq0} \frac{\| ( I - Q_P)\mathbf{v}\|^2}{\|\mathbf{v}\|^2_A},
\]
\end{linenomath}
motivating the often-used simpler form of the WAP,
\[
\|(I - Q_P)\mathbf{v}\|^2 \leq \frac{K_{TG}}{\|A\|} \|\mathbf{v}\|_A^2.
\]

The necessarily complementary role of relaxation and coarse-grid correction
in AMG is accounted for in the WAP by requiring interpolation accuracy with
respect to $\widetilde{M}$, that is the coarse-grid correction must account
for low-eigenvalue modes of $\widetilde{M}$, which are not effectively
reduced through relaxation with $\widetilde{M}$. A bound on two-grid
convergence can also be formulated as two independent constraints based
on coarse-grid selection and bounding the energy in the range of $P$ as
follows \cite{Falgout:2004cs,Vassilevski:2010vy}. 
\begin{lemma}[Two-grid energy-stability]\label{lem:2g_stability}
Let $X$ be spectrally equivalent to $\widetilde{M}$ as in \eqref{eq:spec_equiv},
and define $A_s = S^TAS, X_s = S^TXS$, where
\begin{equation}
\kappa_s \leq \lambda_{\text{min}}(X_s^{-1}A_s) \leq \lambda_{\text{max}}(X_s^{-1}A_s) \leq c_2. \label{eq:cr1}
\end{equation}
If $PR$ is bounded in energy, $\|PR\|_A^2 \leq C$ for some $C$, a WAP in
the $X$-norm, with projection $PR$ is satisfied. These are sufficient
conditions for uniform two-grid convergence, with
$K_{TG} \leq \frac{c_2}{\kappa_s}\|PR\|_A^2$.
\end{lemma}

Lemma \ref{lem:2g_stability} can be seen as an energy-stability constraint
coupled with a compatibility measure of the fine and coarse grids.
Equation \eqref{eq:cr1} measures how well relaxation, $\widetilde{M}$, or
the spectrally equivalent $X$, captures information about the
fine-grid operator, $A_s$. 
This is based on the idea of compatible relaxation \cite{Brandt:2000vn,Brannick:2010hz},
which ensures that
the relaxation scheme is able to effectively reduce error on the fine grid. Then,
assuming a compatible choice of grids, interpolation must be stable in energy,
where $\|PR\mathbf{v}\|_A^2 \leq C \|\mathbf{v}\|_A^2$.
Note, Lemma \ref{lem:2g_stability} is equivalent to Theorem \ref{th:wap}, however
the differing explicit conditions make for different approaches to constructing
multigrid hierarchies. One important difference is that in this case, the lemma
is formulated in terms of a general operator norm as opposed to a constraint
for all $\mathbf{v}$. This distinction is discussed in more detail in Section
\ref{sec:practice}.

\subsection{Multilevel convergence}\label{sec:theory:multi}

Now that two-grid convergence theory has been introduced, let us continue by
considering when two-level convergence can be extended to the multilevel
setting. For two-level convergence, either the weak or strong approximation property provide
sufficient conditions for convergence; however, the same is not true when considering the
multilevel case. In the multilevel setting, if the smoothing and strong
approximation properties hold on all levels of the multigrid
hierarchy, with constants that are uniformly bounded (independently of
the level in the hierarchy), then multilevel convergence of the
multigrid V-cycle can be proven \cite{JWRuge_KStuben_1987a}.  Even if
the weak approximation property holds uniformly, though, multilevel
convergence still cannot be guaranteed.  Since $A(I-\pi_A) =
(I-\pi_A^T)A$, we can easily derive the bound
\[
\|A(I-\pi_A)\vect{v}\|_X \leq \|I-\pi_A^T\|_X\|A\vect{v}\|_X,
\]
showing that if $\|I-\pi_A^T\|_X$ is not uniformly bounded across the
levels in the hierarchy, then a uniform weak approximation property
does not imply a uniform strong approximation property.

The standard example of this is the use of a piecewise constant
interpolation operator, $P$, for any standard discretization of the
Poisson problem.  If $\vect{v}$ is a smooth vector, then
$\|A\vect{v}\|_X$ will be small for many reasonable choices of $X$,
such as $X = D^{-1}$, where $D$ is the diagonal of the system matrix, $A$.  After
coarse-grid correction, $(I-\pi_A)\vect{v}$ will have jumps induced by
the piecewise-constant interpolation, so $\|(I-\pi_A)\vect{v}\|_A$
will be large, reflecting the high-frequency character of
$(I-\pi_A)\vect{v}$.  With this, the strong approximation property can
only be achieved with a large constant, $\beta$.  In contrast, since
$\|A(I-\pi_A)\vect{v}\|_X$ will also be large, the weak approximation
property can be fulfilled with a moderate constant, $\beta$.  As is
well-known, piecewise constant interpolation is sufficient for good
two-level convergence, but not multilevel, consistent with the
theoretical results.

Let $E_{MG}$ be the error-propagation matrix
for a $V(1,1)$-cycle with a full multigrid hierarchy. The resulting convergence
factor is bounded by the $A$-norm of $E_{MG}$, which can take the form
\[
\|E_{MG}\|_A  = 1 - \frac{1}{K_{MG}},
\]
for some $K_{MG} \geq 1$. Multigrid with an arbitrary number of levels can
be thought of as a recursive use of two-grid methods with inexact coarse-grid solves, which is
typically how convergence theory is formulated in the multilevel setting. 
The standard multilevel convergence result and some equivalent or
sufficient conditions are stated below.  In all cases, we assume that
the multigrid hierarchy is specified by matrices $A^{(k)}$ and
interpolation operators $P^{(k)}$, with the convention that $P^{(k)}$
is the interpolation operator from level $k+1$ to level $k$, and
$A^{(k+1)} = \left(P^{(k)}\right)^TA^{(k)}P^{(k)}$.  Furthermore, we
assume that on each level, a relaxation scheme is specified that
satisfies a consistent smoothing property,
\[
\left\|\left(I-\left(M^{(k)}\right)^{-T}A^{(k)}\right)\vect{v}^{(k)}\right\|_{A^{(k)}}^2
\leq \left\|\vect{v}^{(k)}\right\|_{A^{(k)}}^2 - \frac{\alpha}{\|A^{(k)}\|}\left\|A^{(k)}\vect{v}^{(k)}\right\|^2,
\]
for all vectors, $\vect{v}^{(k)}$, with $\alpha$ independent of $k$.

\begin{theorem}[Strong approximation property]
\label{th:sap}
If, for every $\mathbf{v}^{(k)}$, there exists a $\mathbf{v}^{(k+1)}$ such that
\begin{equation}
\left\|\mathbf{v}^{(k)} - P^{(k)}\mathbf{v}^{(k+1)}\right\|_{A^{(k)}}^2 \leq \frac{\beta}{\|A^{(k)}\|} \left\|A^{(k)}\vect{v}^{(k)}\right\|^2,\label{eq:sap}
\end{equation}
for some $\beta$ independent of $k$. Then the multilevel $V(1,1)$-cycle converges uniformly, and
$K_{MG} \leq \beta/\alpha$.
\end{theorem}

Note that the strong approximation property (SAP), is similar to the
multilevel generalization of the WAP of equation
\eqref{eq:wap_bestkind} when $X = \frac{1}{\|A\|}I$,
\[
\left\|\mathbf{v}^{(k)} - P^{(k)}\mathbf{v}^{(k+1)}\right\|^2 \leq \frac{\beta}{\|A^{(k)}\|} \left\|\vect{v}^{(k)}\right\|_{A^{(k)}}^2.
\]
In this form of the WAP, interpolation of an eigenvector, $\mathbf{v}^{(k)}$, must
be accurate in the $l^2$-norm to the order of its corresponding eigenvalue, with
constant $\frac{\beta}{\|A^{(k)}\|}$. A stronger statement is required by the SAP, namely
that interpolation of an eigenvector in the $A^{(k)}$-norm must be accurate to the order of its
corresponding eigenvalue. A
detailed look at the SAP can be found in Theorem 5.6.1
and Chapter 6 of\cite{Vassilevski:2010vy}. 

Two sufficient conditions for the strong approximation property are
stated below.  Both rely on the smoothing property stated above
holding uniformly across all levels.  To simplify notation, we define
$\pi_k =
P^{(k)}\left(A^{(k+1)}\right)^{-1}\left(P^{(k)}\right)^TA^{(k)}$ as
the $A^{(k)}$-orthogonal projection on level $k$ of the hierarchy, and
$Q_k = P^{(k)}\left(\left(P^{(k)}\right)^TP^{(k)}\right)^{-1}\left(P^{(k)}\right)^T$
as the $\ell^2$-orthogonal projection on level $k$ of the hierarchy.
\begin{corollary}[$l^2$-boundedness of $\pi_k$]
\label{cor:sap1}
If, for every $\mathbf{v}^{(k)}$, 
\begin{equation}
\left\|(I - \pi_k)\mathbf{v}^{(k)}\right\|^2 \leq \frac{\beta}{\|A^{(k)}\|^2} \left\|A^{(k)}v^{(k)}\right\|^2\label{eq:wap2},
\end{equation}
for some $\beta$ independent of $k$, then the multilevel $V(1,1)$-cycle converges uniformly, and $K_{MG} \leq \beta/\alpha$.
\end{corollary}
\begin{corollary}[WAP$(A^2)$]
\label{cor:sap2}
If, for every $\mathbf{v}^{(k)}$, 
\begin{align}
\left\|(I - Q_k)\mathbf{v}^{(k)}\right\|^2 & \leq \frac{\beta}{\|A^{(k)}\|^2} \left\|A^{(k)}v^{(k)}\right\|^2,
\end{align}
for some $\beta$ independent of $k$, then the SAP holds with constant
$\beta$ and $K_{MG} \leq \beta/\alpha$.
\end{corollary}
\begin{proof}
For $\hat{\mathbf{v}}^{(k+1)} = \left(\left(P^{(k)}\right)^TP^{(k)}\right)^{-1}\left(P^{(k)}\right)^T\mathbf{v}^{(k)}$,
\[
\left\|\mathbf{v}^{(k)}-P^{(k)}\hat{\mathbf{v}}^{(k+1)}\right\|_{A^{(k)}}^2 \leq \|A^{(k)}\|\left\|\mathbf{v}^{(k)}-P^{(k)}\hat{\mathbf{v}}^{(k+1)}\right\|^2 \leq \frac{\beta}{\|A^{(k)}\|} \left\|A^{(k)}v^{(k)}\right\|^2.
\]
\end{proof}

\subsection{``Optimal'' and ``ideal'' interpolation}

Returning to the two-level case, alongside bounds on two-grid convergence factors, specific interpolation operators have
been derived as the best interpolation operator in certain contexts.
Let $P = \begin{bmatrix}W\\I\end{bmatrix}$ and $R = \begin{bmatrix}0&I\end{bmatrix}$,
and consider relaxing the numerator of \eqref{eq:wap1} from the $\widetilde{M}$-norm to the following problem:
\begin{equation}
\min_P \max_{\mathbf{v}} \frac{\| ( I - PR)\mathbf{v}\|^2}{\|\mathbf{v}\|^2_A} = \begin{bmatrix} -A_{ff}^{-1}A_{fc} \\ I\end{bmatrix}
	:= P_{ideal}, \label{eq:wap_ideal}
\end{equation}
where $P_{ideal}$ is so-called ``ideal interpolation'' \cite{Falgout:2004cs}.
Due to the inverse of $A_{ff}$, $P_{ideal}$ is often a difficult
operator to compute directly and may be dense, neither of which
are compatible with the goals of AMG (see Section \ref{sec:practice}). However,  denoting the graph distance
between nodes $i$ and $j$ in $A_{ff}$ by  $|i-j|_G$, then the
following decay property is well-known:
\[
[A_{ff}^{-1}]_{ij} \leq C q^{|i-j|_G-1},
\]
for some constant $C$ and $q < 1$, where $q\approx \frac{\sqrt{\kappa}-1}{\sqrt{\kappa}+1}$, for condition number, $\kappa$, of
$A_{ff}$ \cite{Brannick:2007fb,DemkoMossSmith}. Thus, for sparse,
well-conditioned $A_{ff}$ (as is expected with
a proper choice of coarse grid), coefficients of $A_{ff}^{-1}$ decay exponentially fast away from the
diagonal. Under this assumption, $A_{ff}^{-1}$ can be approximated well with a sparse matrix, at least in a Frobenius sense.
Approaches to approximating $P_{ideal}$ can be found in \cite{Manteuffel:2016vd,Brannick:2007fb},
each of which contain some variation of the following result:
\begin{lemma}[Ideal interpolation]\label{lem:ideal}
Let $P_{ideal}$ be as in \eqref{eq:wap_ideal} and let $P$ take the form $P = \begin{bmatrix}W\\I\end{bmatrix}$, restricted to a fixed
nonzero sparsity pattern. Then minimizing the difference between columns of $P$ and $P_{ideal}$ in the $A$-norm is equivalent
to minimizing each column of $P$ in the $A$-norm. Furthermore, the solution of this minimization is unique. 
\end{lemma}

However, considering that \eqref{eq:wap_ideal} does not provide a sharp bound on convergence, ideal interpolation typically does
not provide optimal (two-grid) convergence factors over all $P$. The optimal $P$ with respect to two-grid convergence
is given in the following lemma \cite[Lemma 1]{brannick2017optimal}.
\begin{lemma}[Optimal interpolation]\label{lem:opt}
Let $0<\lambda_1\leq ... \leq \lambda_n$ and $\mathbf{v}_1,...,\mathbf{v}_n$
denote the eigenvalues and eigenvectors, respectively, of the generalized
eigenvalue problem
\[
A\mathbf{v} = \lambda\widetilde{M}\mathbf{v}.
\]
Then the minimal convergence rate of the two-grid method $\|E_{TG}(P)\|_A$
over all $P$ with dim$(P) = n_c$ is given by
\[
\|E_{TG}(P_{opt})\|_A^2 = 1 - \lambda_{n_c+1},
\]
with corresponding optimal interpolation matrix given by
\[
P_{opt} = \begin{bmatrix} \mathbf{v}_1 & ... & \mathbf{v}_{n_c}\end{bmatrix}.
\]
\end{lemma}

Although results in \cite{brannick2017optimal} suggest that, at times, a sparse approximation to $P_{\textnormal{opt}}$
may be feasible, it is certainly more difficult to develop a cheap, sparse approximation to $P_{\textnormal{opt}}$ compared with
$P_{\textnormal{ideal}}$. That being said, Lemma \ref{lem:opt} does corroborate the general AMG approach of including
eigenvectors of $A$ (actually of $\widetilde{M}^{-1}A$) associated with small eigenvalues in $\Ima(P)$. In fact, it follows
from Lemma \ref{lem:opt} that if the first $n_c+1$ eigenvalues of $A\mathbf{v} = \lambda\widetilde{M}\mathbf{v}$ 
are all approximately zero, AMG \textit{cannot} achieve strong convergence factors. This highlights the importance of
the distribution of eigenvalues on the performance of AMG. 

\section{Interpolation in practice}\label{sec:practice}

AMG is a popular solver largely because of its linear complexity in
the setup and solve phase.  Let $A$ be SPD and consider forming a multigrid hierarchy to solve $A\mathbf{x} = \mathbf{b}$. Interpolation operators in
AMG methods are often (implicitly) constructed with the goal of
controlling or minimizing some functional with a theoretical
relation to convergence, such as \eqref{eq:wap1}, \eqref{eq:sap}, \eqref{eq:wap2}, or \eqref{eq:wap_ideal}. However, there
are two important factors that must be considered in practice and are generally absent from theory -- (i)  the process used to form
interpolation operators must remain linear in complexity in keeping
with the desired $O(n)$ total cost for AMG methods, and (ii) interpolation operators must
remain sparse in order to construct a sparse coarse-grid matrix that
can be used in a recursive process. 

These constraints can
prove difficult to achieve when designing AMG methods, and make
approximating some of the bounds in Section \ref{sec:theory}
more tractable than others. In particular, the operators used in the convergence theory, such as $\pi_A$, $P_{ideal}$, and $P_{opt}$, generally cannot
be easily computed in any practical setting. Furthermore, these convergence results are typically required to hold for all $\mathbf{v}$ or, equivalently, for some basis for the space
such as the eigenvectors of $A$. Constructing interpolation or coarsening based on a full basis of vectors is generally not tractable in linear
complexity and, thus, two forms of approximation are often used, (i) work with a candidate set of $k$ vectors, where
$k \ll n$, or (ii) work in an operator norm, which is a supremum over all vectors. 

The first approach is to directly satisfy conditions of a theorem but
only for a set of candidate vectors of dimension $k\ll n$. This
is a standard approach for satisfying the WAP or SAP, and classical
AMG \cite{JWRuge_KStuben_1987a} can be viewed as doing this for only
the constant vector, while smoothed aggregation (SA) \cite{Van:2001bw} may
use a larger basis, such as the rigid-body modes for elasticity
problems.  Most adaptive multigrid methods \cite{DAmbra:2013iwa,
  Brezina:2004eh, Brandt:2011hb} can also be viewed in this way, where
the candidate vectors arise from the adaptive process.  As in \eqref{eq:wap1}, the (two-grid) convergence rate is bounded by the
maximum of $K_{TG}$ over all $\mathbf{v} \neq \mathbf{0}$. Noting the denominator of $\|\mathbf{v}\|_A^2$,  the maximum will generally occur for
$\mathbf{v}$ associated with small eigenvalues of $A$, where
$\|\mathbf{v}\|_A^2$ is very small compared to $\|\mathbf{v}\|_{\widetilde{M}}^2$. For differential operators,
it is common to have a zero or near-zero row sum, making the constant a good representation of low-energy modes. Developing and
using additional candidate vectors is the basis for adaptive approaches, which are designed for difficult linear systems beyond the scope of
classical SA or AMG \cite{Brandt:2011hb, Brezina:2004eh,DAmbra:2013iwa}. In such solvers, an adaptive process is used to develop
a set of target vectors representative of low-energy modes of $A$. Interpolation is then constrained to interpolate these
modes either exactly or nearly so, and the process is repeated on coarse grids.

An alternative approach is to formulate the
minimization over all vectors $\mathbf{v}\neq\mathbf{0}$. Consider the
energy-stability constraint in Lemma \ref{lem:2g_stability},
based on controlling $\|PR\|_A^2$. As an induced $A$-norm, $\|PR\|_A$
is defined via a supremum over all $\mathbf{v}\neq\mathbf{0}$. However, $\|PR\|_A$
can be bounded using the Frobenius norm, which gives an indirect
approach to bounding $\|PR\|_A$. Let $R
= \begin{bmatrix}0&I\end{bmatrix}$, then
\[
\|PR\|_A^2 = \|A^{\frac{1}{2}}PRA^{-\frac{1}{2}}\|^2 \leq \|A^{\frac{1}{2}}PRA^{-\frac{1}{2}}\|_F^2 = \tr(P^TAPS_A),
\]
where $S_A^{-1}:= RA^{-1}R^T = (A_{cc} - A_{cf}A_{ff}^{-1}A_{fc})^{-1}$,
is the inverse of the Schur complement of $A$ in $A_{cc}$. Although an interesting
equivalence, the Schur complement is difficult to form in practice. A more
tractable approach is obtained by using an intermediate bound,
\[
\|PR\|_A^2 \leq \|A^{-1}\|\|A^{\frac{1}{2}}PR\|^2 \leq \|A^{-1}\|\|A^{\frac{1}{2}}PR\|_F^2 = \|A^{-1}\|\tr(R^TP^TAPR).
\]
Given the form of $R = (\mathbf{0}, I)$, this simplifies to
\begin{equation}
\|PR\|_A^2 \leq \|A^{-1}\|\tr(P^TAP). \label{eq:ainv}
\end{equation}
Minimizing $\tr(P^TAP)$ was proposed in this form in \cite{Brannick:2007fb}, and is equivalent to minimizing columns of $P$
in the $A$-norm. This approach has been used in smoothed aggregation (SA) \cite{Van:2001bw}, root-node AMG \cite{Manteuffel:2016vd,
Schroder:2012di}, and the general energy-minimization framework proposed in \cite{Olson:2011fg}. Recall from Lemma \ref{lem:ideal} that
minimizing energy in columns of $P$ is also equivalent to minimizing the difference between columns of $P$ and $P_{ideal}$ in the $A$-norm. 

It is worth considering the leading constant in \eqref{eq:ainv}, $\|A^{-1}\| = \frac{1}{\lambda_{\text{min}}(A)}$, as this is likely large and
could lead to a poor bound on $\|PR\|_A^2$. Note that the energy constraint in Lemma \ref{lem:2g_stability} can also be formulated as
\[
\mathbf{v}^TR^TP^TAPR\mathbf{v} \leq \eta \mathbf{v}^TA\mathbf{v} \hspace{3ex}
\Longleftrightarrow \hspace{3ex} \mathbf{v}_c^TA_c\mathbf{v}_c \leq \eta\mathbf{v}^TA\mathbf{v},
\]
for all vectors $\mathbf{v}$ with $\mathbf{v}_c = R\mathbf{v}$
\cite[Theorem 5.2]{Falgout:2005hm}. The factor of $\|A^{-1}\|$ in
\eqref{eq:ainv} accounts for the possibility that $P$ is chosen so
that $\mathbf{v}_c^TA_c\mathbf{v}_c$ is an $O(1)$ quantity when
$\mathbf{v}$ corresponds to the smallest eigenvalue of $A$
(corresponding to a bad choice of $P$). In practice, we make
the heuristic assumption that the choice of $P$ will not be so bad
and, thus, minimizing $\tr(P^TAP) =
\sum_i \lambda_i(A_c)$ is an effective way to control $\|PR\|_A$.

\section{Trace-minimization}\label{sec:functional}

As discussed in Section \ref{sec:practice} and can be seen in other AMG methods, AMG interpolation operators are often constructed
based on some combination of ensuring that a given set of candidate
vectors is interpolated exactly, while ensuring energy stability of the
coarse-grid operator. In this direction, we now propose to form $P$ through minimizing a general weighted functional combining these two approaches,
\begin{equation}
\mathcal{G}(P) = (1-\tau)\frac{\|(I - PR)\mathbf{v}\|_{\widetilde{M}}^2}{\|\mathbf{v}\|_A^2} + \tau \tr(P^TAP),
\label{eq:gen_func0}
\end{equation}
for $\tau\in[0,1)$ and candidate vector $\mathbf{v}$. If multiple candidate vectors, $\{\mathbf{v}_i\}$, are available a priori, for example,
the rigid body modes in elasticity, then we minimize over the maximum
$\mathbf{w}\in$ Span$\{\mathbf{v}_i\}$:
\[
\mathcal{G}(P) =
(1-\tau)\max_{\mathbf{w}\in\text{Span}\{\mathbf{v}_i\}\setminus\{\mathbf{0}\}}\frac{\|(I - PR)\mathbf{w}\|_{\widetilde{M}}^2}{\|\mathbf{w}\|_A^2} + \tau \tr(P^TAP),
\]
This is a complementary
approach, focusing on achieving accurate interpolation of the
low-energy modes in the candidate set as well as energy stability on the coarse grid. It is also complementary
in the sense that the first term is defined over a candidate set of vectors, $\{\mathbf{v}_i\}$, while the second term is defined over $P$,
and should improve interpolation regardless of the provided candidate vectors. 

\begin{linenomath}
Let $P$ take the form $P = \begin{bmatrix}W\\I\end{bmatrix}$, and consider minimizing \eqref{eq:gen_func0}. Define a set of $n_B$
candidate vectors as $A$-orthonormalized columns of a matrix $B = \begin{bmatrix}B_f\\B_c\end{bmatrix}$, and let $X\simeq\widetilde{M}$ as in \eqref{eq:spec_equiv}.
Then, consider minimizing $K_{TG}$ from \eqref{eq:wap1}, restricted to
unit linear combinations of $\mathbf{v}\in \mathcal{B} =
\{\mathbf{w}\in\Ima(B)\mid \|\mathbf{w}\|_A=1\}$:
\begin{align}
\max_{\mathbf{v\in\mathcal{B}}} \|(I - \pi_{\widetilde{M}})\mathbf{v}\|_{\widetilde{M}}^2 & \leq 
	c_2\max_{\mathbf{v\in\mathcal{B}}}\|(I - \pi_X)\mathbf{v}\|_{X}^2 \nonumber\\
& \leq c_2 \max_{\mathbf{v\in\mathcal{B}}} \|(I - PR)\mathbf{v}\|_{X}^2 \nonumber\\
& = c_2 \max_{\mathbf{v\in\mathcal{B}}} \Big\langle X \begin{bmatrix} \mathbf{v_f} - W\mathbf{v_c} \\ 0 \end{bmatrix},
	\begin{bmatrix} \mathbf{v_f} - W\mathbf{v_c} \\ 0 \end{bmatrix} \Big\rangle \nonumber\\
& = c_2 \max_{\mathbf{v\in\mathcal{B}}} \Big\langle X_{ff}W\mathbf{v_c}, W\mathbf{v_c}-2\mathbf{v}_f \Big\rangle + 
	c_2\|\mathbf{v_f}\|_{X_{ff}}^2 \nonumber\\
& \leq c_2 \Big\langle X_{ff}WB_c, WB_c-2B_f \Big\rangle_F + c_2\|B_f\|_{X_{ff}}^2 \nonumber\\
& = c_2\Big\langle X_{ff}WB_cB_c^T, W\Big\rangle_F - 2c_2\Big\langle X_{ff}W, B_fB_c^T\Big\rangle_F + c_2\|B_f\|_{X_{ff}}^2. \label{eq:term1}
\end{align}
This approximates the WAP in the $X$-norm using an $l^2$-projection onto $\Ima(P)$ (as opposed to the optimal $\pi_{\widetilde{M}}$-orthogonal
projection). Recall the second term in \eqref{eq:gen_func0} corresponds to minimizing the columns of $P$ in the $A$-norm. Expanding
$\tr(P^TAP)$ gives
\begin{align}
\tr(P^TAP) & = \tr\bigg(\begin{bmatrix}W^T & I\end{bmatrix} \begin{bmatrix} A_{ff} & A_{fc} \\ A_{cf} & A_{cc}\end{bmatrix} \begin{bmatrix} W \\ I\end{bmatrix}\bigg) \nonumber\\
	& = \tr(W^TA_{ff}W) + 2\tr(A_{cf}W) + \tr(A_{cc}) \nonumber \\
	& = \Big\langle A_{ff}W,W\Big\rangle_F + 2\Big\langle W, A_{fc}\Big\rangle_F + \tr(A_{cc}) \label{eq:term2}
\end{align}
\end{linenomath}

\begin{linenomath}
Substituting equations \eqref{eq:term1} and \eqref{eq:term2} into \eqref{eq:gen_func0} gives a functional of $W$ to minimize in forming
$P$. Dropping terms independent of $W$ and pulling out a factor of two for a more familiar form, define
\begin{align}
\begin{split}
\mathcal{F}(W) & = \frac{\tau}{2}\Big\langle A_{ff}W,W\Big\rangle_F + \frac{c_2(1-\tau)}{2}\Big\langle X_{ff}WB_cB_c^T, W\Big\rangle_F \\
	&\hspace{10ex} - \Big\langle W, c_2(1-\tau)X_{ff}B_fB_c^T - \tau A_{fc}\Big\rangle_F. \label{eq:functional}
\end{split}
\end{align}
Observe that \eqref{eq:functional} is a quadratic functional in $W$. Define a bounded linear operator, $\mathcal{L}$, and right-hand-side,
$\mathcal{B}$, as
\begin{align}
\mathcal{L}W & = \tau A_{ff}W +  c_2(1-\tau)X_{ff}WB_cB_c^T \label{eq:linop} \\
\mathcal{B} & = c_2(1-\tau)X_{ff}B_fB_c^T - \tau A_{fc}, \notag
\end{align}
in which case $\mathcal{F}(W) = \tfrac{1}{2}\langle \mathcal{L}W, W\rangle_F - \langle W,\mathcal{B}\rangle_F$. Note that
if $A_{ff}$ and $X_{ff}$ are symmetric and positive definite, then $\mathcal{L}$ is self-adjoint and positive definite in the Frobenius norm:
\begin{align*} 
\Big\langle \mathcal{L}W, Z\Big\rangle_F & = \Big\langle \tau A_{ff}W, Z\Big\rangle_F + c_2(1-\tau)\Big\langle X_{ff}WB_cB_c^T, Z\Big\rangle_F \\
& = \Big\langle \tau W, A_{ff}Z\Big\rangle_F + c_2(1-\tau)\Big\langle W, X_{ff}ZB_cB_c^T \Big\rangle_F \\
& = \tau\Big\langle W, \mathcal{L}Z\Big\rangle_F, \\
\Big\langle \mathcal{L}W,W\Big\rangle_F & = \tau\Big\langle A_{ff}W, W\Big\rangle_F + c_2(1-\tau)\Big\langle X_{ff}WB_cB_c^T, W\Big\rangle_F \\
& = \tau\Big\langle A_{ff}W, W\Big\rangle_F + c_2(1-\tau)\Big\langle X_{ff}WB_c,WB_c\Big\rangle_F \\
& > 0 \text{ when }W\neq 0.
\end{align*}
Using the symmetry of $\mathcal{L}$, the first and second Frech\'et derivative of $\mathcal{F}$ are given by:
\begin{align*}
\mathcal{F}'(W)[V] & = \lim_{\alpha\to 0} \frac{\mathcal{F}(W + \alpha V) - \mathcal{F}(W)}{\alpha} \\
& = \Big\langle \mathcal{L}W - \mathcal{B}, V\Big\rangle_F \\
\mathcal{F}''(W)[V][U] & = \Big\langle \mathcal{L}U, V\Big\rangle_F.
\end{align*}
\end{linenomath}
Since $\mathcal{L}$ is self-adjoint and positive definite, $\mathcal{F}''(W)[V] \geq 0$ $\forall$ $V$. Thus, the minimum of $\mathcal{F}$
in $W$ is achieved at $W$ such that $\mathcal{F}'(W)[V] = 0$, and $\mathcal{F}'(W) = 0$ $\forall$ $V$ if and only if
$\mathcal{L}W = \mathcal{B}$. This has a unique solution, $W = \mathcal{L}^{-1}\mathcal{B}$. However, it is likely that
$\mathcal{L}^{-1}\mathcal{B}$ is dense and not practical, motivating a constrained sparsity pattern for $W$.

\subsection{Constrained sparsity pattern}

In practice, the sparsity pattern of $W$ must be fixed a priori in order to control the operator complexity of $W$ and $A_c$.
Define a vector space
\[
\mathcal{X} = \Big\{ W \text{ : } W\in\mathbb{R}^{N_f\times N_c}, W_{ij} = 0 \text{ if } (i,j)\not\in \mathcal{N}\Big\},
\]
for a set of indices $\mathcal{N}$ denoting a fixed sparsity pattern for $W$. A Hilbert space $\mathcal{H}$ can be defined
over $\mathcal{X}$ with the Frobenius inner product, $\langle A,B\rangle_F = \sum_{ij} A_{ij}B_{ij}$. It is easily verified that
$\mathcal{X}$ is complete over the norm induced by $\langle\cdot,\cdot,\rangle_F$ due to the completeness of $\mathbb{R}$.
Now define the bounded linear functional $\hat{\mathcal{L}}:\mathcal{H}\to\mathcal{H}$ as 
\[
(\hat{\mathcal{L}}W)_{ij} =
\begin{cases} (\mathcal{L}W)_{ij} & (i,j)\in\mathcal{N} \\ 0 & (i,j)\not\in\mathcal{N}\end{cases},
\]
and a corresponding bilinear form
\[
a(W,V) = \Big\langle \hat{\mathcal{L}}W, V\Big\rangle_F.
\]
A quadratic form as in \eqref{eq:functional} restricted over $\mathcal{N}$ can then be defined as 
\begin{equation}
\hat{\mathcal{F}}(W) = \frac{1}{2}\Big\langle \hat{\mathcal{L}}W, W\Big\rangle_F -  \Big\langle W, \hat{\mathcal{B}}\Big\rangle_F,
\label{eq:sp_func}
\end{equation}
where $\hat{\mathcal{B}}\in\mathcal{H}$ is $\mathcal{B}$ restricted to $\mathcal{N}$. Note that in $\mathcal{H}$,
$\langle W,\mathcal{B}\rangle_F = \langle W,\hat{\mathcal{B}}\rangle_F$. A similar derivation as shown for $\mathcal{L}$ confirms
that $\hat{\mathcal{L}}$ is self-adjoint and $a(W,V)$ symmetric. Then, observe that for $W\in\mathcal{H}, W\neq 0$,
$\hat{\mathcal{L}}$ and $a(W,V)$ are positive:
\[
\Big\langle \hat{\mathcal{L}}W,W\Big\rangle_F = \Big\langle \mathcal{L}W,W\Big\rangle_F > 0,
\]
The following standard lemma of functional analysis can then be
invoked to find a minimizer of $\hat{\mathcal{F}}(W)$ \eqref{eq:sp_func}.
\begin{lemma}
  \label{lem:solution}
  \begin{linenomath}
Let $a(x,y)$ be a bounded, symmetric, positive-definite bilinear form
on a Hilbert space $\mathcal{H}$, and $\mathcal{G}(x)$ be a
bounded linear functional on $\mathcal{H}$. Then the following are equivalent
\begin{align}
x & = \min_{x\in\mathcal{H}}\hspace{1ex} \frac{1}{2}a(x,x) - \mathcal{G}(x) + C \label{eq:lem1_1} \\
x & \textnormal{ satisfies } a(x,y) = \mathcal{G}(y) \hspace{1ex} \text{ for all }y\in\mathcal{H}
\label{eq:lem1_2}
\end{align}
\end{linenomath}
Furthermore, there exists a unique solution $x\in\mathcal{H}$ satisfying
\eqref{eq:lem1_1}, \eqref{eq:lem1_2}.
\end{lemma}

Based on Lemma \ref{lem:solution}, we seek the unique solution to
\begin{equation}
\hat{\mathcal{L}}W = \hat{B}, \hspace{2ex}W\in\mathcal{H},\label{eq:sp_solution}
\end{equation}
which can be iterated towards using the preconditioned conjugate gradient method
introduced in Section \ref{sec:pcg}.

\begin{remark}
A conceptual limiting case of the proposed weighted energy minimization is to interpolate candidate vectors exactly
and minimize energy based on that constraint. However, this does not directly fit into the framework of \eqref{eq:gen_func0}.
Constrained energy-minimization has been proposed in various forms \cite{Brannick:2007fb,Mandel:1999wg,Olson:2011fg,Wan:1999ky},
and was used as a basis for root-node AMG in \cite{Manteuffel:2016vd}. Defining the affine space
$\mathcal{A} = \{ W \text{ : } W\in\mathcal{H}\text{ and } WB_c = B_f\}$, the constrained minimization problem is given by 
\begin{equation}
W  = \argmin \Big\langle A_{ff}W + A_{fc},W\Big\rangle_F, \hspace{1ex} W\in\mathcal{A}.\label{eq:constrained}
\end{equation}
Since the linear operator now consists of normal matrix multliplication, $A_{ff}W$ as opposed to the left and right multiplication
in \eqref{eq:linop}, the existence and uniqueness of a solution to \eqref{eq:constrained} can be shown in a linear algebra
setting (see \cite{Olson:2011fg}) along with a CG implementation based on projecting into $\mathcal{A}$.
\end{remark}

\section{Preconditioned conjugate gradient}\label{sec:pcg}

Because $\hat{\mathcal{L}}$ is self-adjoint and positive in $\mathcal{H}$, conjugate gradient (CG) in the Hilbert space setting is a
competitive approach to solving \eqref{eq:sp_solution} in an iterative fashion. It is generally advisable to precondition CG iterations
for optimal convergence. Here, we construct a diagonal preconditioner for \eqref{eq:sp_solution} to make iterations more robust
when $\hat{\mathcal{L}}$ is poorly conditioned at a marginal increase of computational cost.

Unlike with matrices, however, it is not clear what the ``diagonal'' of $\hat{\mathcal{L}}$ is. Let $W\in\mathbb{R}^{N_f\times N_c}$, and
define the operator $\overline{(W)}$ as the columns of $W$ stacked in a column-vector. Note that $\overline{(W^T)}$ then gives the
rows of $W$ stacked as a column-vector. Let $Y$ be the permutation matrix such that  $\overline{(W)} = Y\overline{(W^T)}$ and
$YY^T = Y^TY = I$, which can be thought of as a mapping of $W$ from row-major format to column-major format. First note the following
lemma with regards to Kronecker products and the action of $Y$. 
\begin{lemma}\label{lem:kron}
Let $Y$ be a permutation matrix mapping $W\in\mathbb{R}^{N_f\times N_c}$ from row-major format to column-major format, that is,
$\overline{(W)} = Y\overline{(W^T)}$. Then, for any $P\in\mathbb{R}^{N_f\times N_f}$ and $Q\in\mathbb{R}^{N_c\times N_c}$,
\[
Y(P\otimes Q)Y^T = Q\otimes P
\]
\end{lemma}
\begin{proof}
\begin{linenomath}
  First consider the structure of $Y$. Note the following relations between $W,
\overline{(W)}$, and $\overline{(W^T)}$, i.e. $W$ stored as a standard dense
matrix, a column-major matrix, and a row-major matrix, respectively,
\begin{align*}
\overline{(W)}_{i+jN_f} & = W_{ij} \\
\overline{(W^T)}_{j+iN_c} & = W_{ij}.
\end{align*}
Defining $Y$ such that $Y\overline{(W^T)} = \overline{(W)}$, it
follows that
\[
Y_{i+jN_f, j+iN_c} = 1, \hspace{2ex}\text{ for $i\in[0,N_f],j\in[0,N_c]$},
\]
and the action of $YAY^T$ is then given as
\begin{equation}
[YAY^T]_{i+jN_f,k+lN_f} = A_{j+iN_c, l+kN_c}.\label{eq:lem_y}
\end{equation}
Now consider the element-wise Kronecker products of $P$ and $Q$:
\begin{align}
[P\otimes Q]_{j+iN_c, l+kN_c} = P_{ik}Q_{jl}, \label{eq:kron1} \\
[Q\otimes P]_{i+jN_f, k+lN_f} = P_{ik}Q_{jl}, \label{eq:kron2}
\end{align}
for $i,k\in[0,N_f], j,l\in[0,N_c]$. Combining \eqref{eq:lem_y}, \eqref{eq:kron1}, and \eqref{eq:kron2} gives
\begin{align*}
[Y(P\otimes Q)Y^T]_{i+jN_f,k+lN_f} & = (P\otimes Q)_{j+iN_c,l+kN_c} \\
& = P_{ik}Q_{jl} \\
& = [P\otimes Q]_{i+jN_f, k+lN_f}.
\end{align*}
It follows that $Y(P\otimes Q)Y^T = Q\otimes P$.
\end{linenomath}
\end{proof}

\begin{remark}
Lemma \ref{lem:kron} is a known result that we arrived at inadvertently, where $Y$ is known as the ``Perfect Shuffle'' matrix
\cite{Davio:1981gc}. Its relation to row-major and column-major
storage of matrices is, to our knowledge, a new contribution to the literature. 
\end{remark}

\begin{linenomath}
Now consider finding the diagonal of $\mathcal{L}$ by looking at $\overline{\mathcal{L}}$ as an operator on $\overline{(W)}$. To
do so, represent the action of $A_{ff}W$ through $(I_{N_c}\otimes A_{ff})\overline{(W)}$, where $(I_{N_c}\otimes A_{ff})$
gives a block diagonal matrix of $N_c$ $A_{ff}$'s, each to be multiplied by one column of $W$. Recalling the identity
$(A\otimes B)(C\otimes D) = (AB\otimes CD)$ and Lemma \ref{lem:kron},
\begin{align*}
\overline{\mathcal{L}}\overline{(W)} & = \tau(I_{N_c}\otimes A_{ff})\overline{(W)} + c_2(1-\tau)(I_{N_c}\otimes X_{ff})\overline{(WB_cB_c^T)} \\
& = \tau(I_{N_c}\otimes A_{ff})\overline{(W)} + c_2(1-\tau)(I_{N_c}\otimes X_{ff})YY^T\overline{(WB_cB_c^T)} \\
& = \tau(I_{N_c}\otimes A_{ff})\overline{(W)} + c_2(1-\tau)(I_{N_c}\otimes X_{ff}) Y \overline{(B_cB_c^TW^T)} \\
& = \tau(I_{N_c}\otimes A_{ff})\overline{(W)} + c_2(1-\tau)(I_{N_c}\otimes X_{ff})Y(I_{N_f}\otimes B_cB_c^T)Y^TY\overline{(W^T)} \\
& = \tau(I_{N_c}\otimes A_{ff})\overline{(W)} + c_2(1-\tau)(I_{N_c}\otimes X_{ff})(B_cB_c^T\otimes I_{N_f})\overline{(W)} \\
& = \tau(I_{N_c}\otimes A_{ff})\overline{(W)} + c_2(1-\tau)(B_cB_c^T\otimes X_{ff})\overline{(W)} \\
& = \Big[ \tau(I_{N_c}\otimes A_{ff}) + c_2(1-\tau)(B_cB_c^T\otimes X_{ff})\Big] \overline{(W)}.
\end{align*}
\end{linenomath}
This derivation can be naturally extended to $\mathcal{H}$, where $W\in\mathcal{H}$ has a specified sparsity pattern, $\mathcal{N}$, by
setting the $k$th row and column of $\overline{\mathcal{L}}$ equal to zero for all $k$ such that  $\overline{(W)}_k := W_{ij}$, and
$(i,j)\not\in\mathcal{N}$. Because $\overline{\mathcal{L}}$ is a block operator with block size $N_f\times N_f$, it follows that there is a
distinct ``diagonal'' in $\mathcal{L}$ corresponding to each $j$th column of $W$,
\begin{equation}
D_j = \tau\cdot\text{diag}(A_{ff}) + c_2(1-\tau)(B_cB_c^T)_{jj}\cdot\text{diag}(X_{ff}).\label{eq:d1}
\end{equation}
A diagonal preconditioning for $\hat{\mathcal{L}}$ is then given by taking the Hadamard product with $\mathcal{D}\in\mathcal{H}$,
where the $j$th column of $\mathcal{D}$ is given by the element-wise inverse of \eqref{eq:d1}:
\begin{equation}
\mathcal{D}_{ij} = \frac{1}{\tau(A_{ff})_{ii} + c_2(1-\tau)(B_cB_c^T)_{jj}(X_{ff})_{ii}}, \hspace{3ex} \text{ for } (i,j)\in\mathcal{N}. \label{eq:diag_pre}
\end{equation}
In the case of $A_{ff}$ having a constant or near-constant diagonal, and letting $X_{ff}$ be the diagonal of $A_{ff}$ (a common
practical choice), $\mathcal{D}$ is constant or near-constant. In practice, preconditioning with $\mathcal{D}$ is important for problems
in which diagonal elements of $A$ or target vectors $B$ consist of a wide range of values.

\subsection{Sylvester and Lyapunov equations}\label{sec:pcg:syl}

In fact, \eqref{eq:diag_pre} can be used to define a preconditioner for general systems of the Sylvestor- or Lyapunov-type: 
\begin{equation}\label{eq:gen_mat}
AWB + CWD = F,
\end{equation}
for solution matrix $W$, where $A,B,C$ and $D$ need not be symmetric (of course an appropriate Krylov solver must be chosen based
on properties of the functional). A diagonal preconditioner for \eqref{eq:gen_mat} is given by taking the Hadamard product with
\begin{equation}\label{eq:genp}
\widehat{\mathcal{D}}_{ij} = \frac{1}{B_{jj}A_{ii} + D_{jj} C_{ii}}.
\end{equation}

Systems of the form in \eqref{eq:gen_mat} arise often in the context of optimal control theory. Letting $B = C = I$, \eqref{eq:gen_mat} is
a Sylvester equation; letting $B = A^T$, $C = -I$, and $D = I$, \eqref{eq:gen_mat} is a discrete Lyapunov equation; and letting $B = C = I$
and $D = A^T$, \eqref{eq:gen_mat} is a continuous Lyapunov equation. There have been many efforts at developing Krylov methods and
preconditioners for such systems; for example, see \cite{Dehghan:2010dx,Ding:2010dc,Hochbruck:1995kj,Simoncini:2009ig,Xie:2010cl}.
Here we develop a simple preconditioner for problems of the form \eqref{eq:genp}, that is easy to construct and apply. 

\section{Numerical results}\label{sec:numerical}

In this section, we present numerical results for a variety of problems, comparing a weighted energy minimization and constrained
energy minimization, and analyzing the choice of constraint vector. The method proposed here is implemented in the PyAMG library
\cite{Bell:2008}; AMG methods such as strength-of-connection, coarsening, etc., follow that of \cite{Manteuffel:2016vd}, and the reader
is referred there for details. In figures, \textit{RN} refers to a constrained energy minimization using root-node AMG
\cite{Manteuffel:2016vd} and TM$_{10^k}$ refers to weighted energy minimization proposed here with weight $\tau = 10^k$. The test
problems considered are: 
\begin{enumerate}\setlength\itemsep{1em}
\item \underline{Anisotropic diffusion:} 2-dimensional rotated anisotropic diffusion, discretized with linear finite elements,
on an unstructured triangular mesh:
  
\begin{align}
   \label{eqn:diffusion}
-\nabla\cdot Q^T D Q\nabla u & = f\quad\textnormal{for ${\Omega} = [0,1]^2$},\\
   \label{eqn:diffusionbdy}
                           u & = 0\quad\textnormal{on $\partial\Omega$},
\end{align}
where
\begin{equation*}
Q = \begin{bmatrix} \cos{\theta} &          - \sin{\theta} \\
                    \sin{\theta} & \phantom{-}\cos{\theta}
    \end{bmatrix},
\qquad
D = \begin{bmatrix}1 & 0 \\ 0 & \epsilon \end{bmatrix}.
\end{equation*}
Due to the unstructured mesh, all angles $\theta\in(0,\sfrac{\pi}{2})$ are effectively equivalent
from a solver perspective; thus, moving forward we (arbitrarily) let $\theta = \sfrac{3\pi}{16}$. Mesh
spacing is taken to be $h\approx 1/1000$, resulting in approximately 1.25M DOFs.

\item \underline{Diffusion with an oscillatory coefficient:} 2-dimensional diffusion problem (as in equations
(\ref{eqn:diffusion}-\ref{eqn:diffusionbdy})), discretized with linear finite elements on a structured, regular triangular
mesh with $2N^2$ elements, with a piecewise linear coefficient that oscillates at every other grid point, regardless of mesh size:
\[
   Q = I,
   D = \begin{bmatrix} f(x,y) &  0 \\
                          0   & f(x,y) 
       \end{bmatrix},
   \]
   where
   \[
   f(x,y) = 
   \begin{cases}
      K    & \mbox{if mod}(Nx,2)=1 \mbox{ AND mod}(Ny,2)=0  \\
      K    & \mbox{if  mod}(Nx,2)=0 \mbox{ AND mod}(Ny,2)=1 \\
      1    & \mbox{if  mod}(Nx,2)=1 \mbox{ AND mod}(Ny,2)=1 \\
      1    & \mbox{if  mod}(Nx,2)=0 \mbox{ AND mod}(Ny,2)=0
   \end{cases}.
\]
Taking $h = 1/N$, we recognize $Nx = x/h$ as the (integer) index of a
mesh point in the $x$-direction, with a similar interpretation of
$Ny$.  For points $(x,y)$ not on the mesh, $f(x,y)$ is interpolated
linearly (as a function in the finite-element space).
This results in a number of coefficient oscillations that
grows proportionally with the number of mesh points, resulting in a checkerboard-like
pattern with alternating large and small coefficients.  Diffusion problems
with large and frequent coefficient changes, such as this one,
traditionally make difficult test problems for  multigrid methods. 

\end{enumerate}

\subsection{Determining $P$}\label{sec:results:P}

Here we look at AMG convergence as a function of the number of
iterations of CG used to determine $P$ and how constraint
vectors are enforced, either exactly, in a constrained energy minimization, or weighted by $\tau\in(0,1)$, in
a weighed energy minimization. For all results, a V-cycle is applied with two iterations of Jacobi pre- and post-relaxation
as a preconditioner for CG. Unless otherwise specified, the constraint vector is chosen as the constant
vector with several Jacobi smoothing iterations applied; weighted energy-minimization uses the 
diagonal preconditioning of Section \ref{sec:pcg}; and constrained energy-minimization uses the diagonal 
preconditioning of \cite{Olson:2011fg}.

\subsubsection{Anisotropic Diffusion}

Figure \ref{fig:poisson} shows the work-per-digit-of-accuracy (WPD) as
a function of the number of iterations of CG used to determine $P$,
for variations in energy minimization applied to anisotropic Poisson (Problem \# 1), with anisotropy $\epsilon \in\{1,0.001,0\}$.
WPD is defined as
\[
   WPD = \frac{-C}{\log_{10}(\rho)},
\]
where $C$ is the cycle-complexity of the multigrid solver and $\rho$ is the average convergence rate of the solver over
all iterations. This metric measures how many work-units, defined as the floating point operations to perform a single matrix-vector
multiply, are required to reduce the residual by one order of magnitude.  This metric is particularly useful for cross-comparisons
of solvers with differing sparsity structures. For more detail, see, for instance, \cite{Manteuffel:2016vd}.

Interpolation is fixed to use a degree-four sparsity pattern, that is, the sparsity pattern for each column of $P$ reaches out to neighbors within
graph distance four from the corresponding C-point (see \cite{Manteuffel:2016vd,Schroder:2012di}). This wider sparsity pattern often
leads to better convergence rates for difficult problems \cite{Schroder:2012di}, but also requires more iterations of energy-minimization.
Essentially, wider sparsity patterns create more interpolation coefficients in $P$, which are then determined through energy-minimization.

\begin{figure}[!ht]
  \centering
  \begin{subfigure}[b]{0.32\textwidth}
    \includegraphics[width=\textwidth]{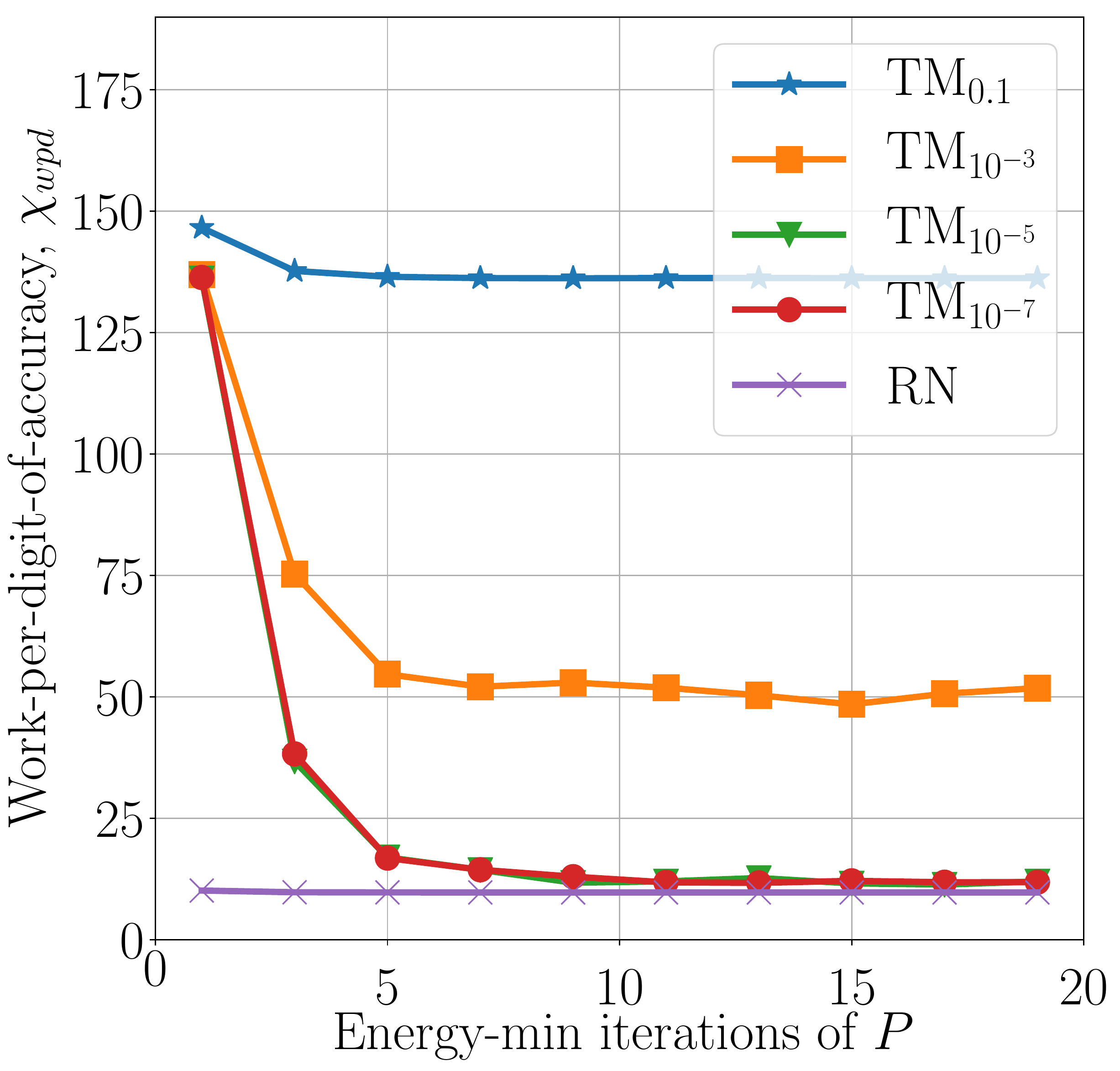}
    \caption{$\epsilon = 1$}
  \end{subfigure}
   \begin{subfigure}[b]{0.32\textwidth}
    \includegraphics[width=\textwidth]{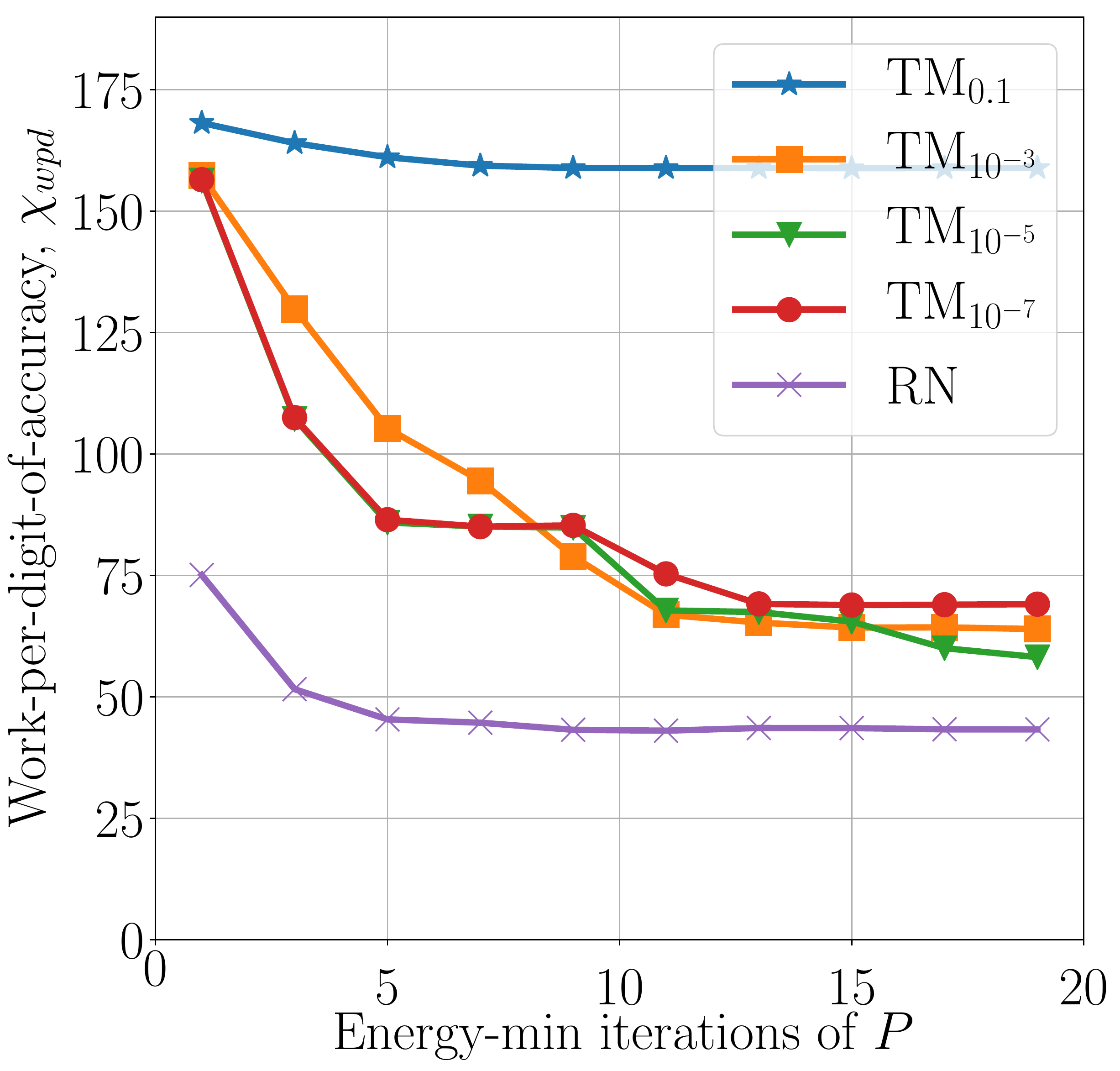}
    \caption{$\epsilon = 0.001$}
  \end{subfigure}
  \begin{subfigure}[b]{0.32\textwidth}
    \includegraphics[width=\textwidth]{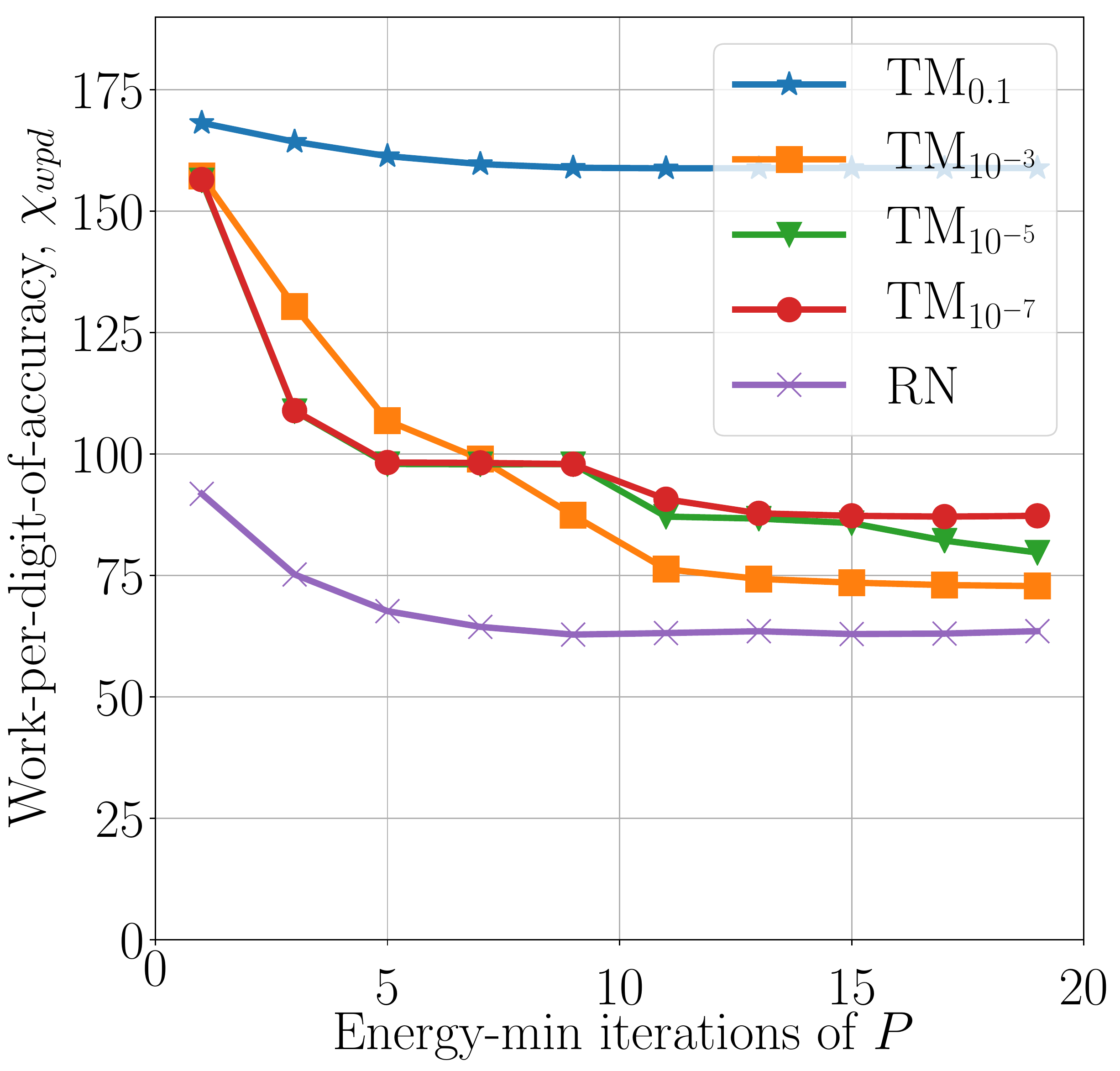}
    \caption{$\epsilon = 0$}
  \end{subfigure}
      \caption{WPD as a function of number of iterations of energy-minimization applied to $P$ for problem $\#1$ and (a) isotropic diffusion
      ($\epsilon=1$), (b) anisotropic diffusion ($\epsilon = 0.001$), and (c) totally anisotropic diffusion ($\epsilon = 0$).}
  \label{fig:poisson}
\end{figure}

Several immediate results follow from Figure \ref{fig:poisson}. First, there is a limit at which additional
 iterations to determine $P$ no longer improve convergence. For the isotropic case $(\epsilon = 1)$, the best convergence
rates are obtained by simply enforcing the constraint with a single constrained smoothing pass; additional
energy-minimization steps do not improve convergence. As the level of anisotropy increases ($\epsilon\to 0$), the number
of iterations of CG required to achieve the best performance increases. However, convergence of the
AMG solver based on a given constraint vector and coarsening scheme remains bounded below, regardless of further
energy minimization of $P$. Second, it is clear that enforcing the constraint exactly or near-exactly is fundamental
to good convergence, even for the simplest isotropic problem. Although theory tells us that interpolating low-energy
modes is necessary for good convergence, the fact that this cannot be achieved through weighted energy minimization
is slightly non-intuitive. Energy-minimization reduces the columns of $P$ in the $A$-norm, which should thus build $P$
to include low-energy modes in its range. Heuristically, it seems that after a handful of CG iterations, the range of $P$
would contain sufficient low-energy modes for good convergence. However, it is clear in Figure \ref{fig:poisson} that even
in the isotropic case, using a large $\tau = 0.1$ to focus on energy minimization over constraints leads to very poor
performance.

Together, these points underline the role of energy minimization in AMG convergence as an acceleration
technique. For some difficult problems, energy minimization is critical to achieving scalable convergence. 
Strongly anisotropic diffusion is one such example that typically proves difficult for standard AMG methods,
but can be solved effectively with constrained energy minimization \cite{Manteuffel:2016vd}. Nevertheless,
regardless of energy minimization, strong convergence cannot be obtained without enforcing or nearly-enforcing
an appropriate constraint vector (Figure \ref{fig:poisson}).

\subsubsection{Diffusion with an oscillatory coefficient}

Figure \ref{fig:jumppoisson} shows the WPD as a function of the number
of iterations of CG used to determine $P$ for variations in energy
minimization applied to the oscillating coefficient problem (Problem \# 2), with coefficient oscillations of $K = 10^6$ and $K=10^3$.  
For $K = 10^6$, Figure \ref{fig:jumppoisson:left} shows results for diagonal preconditioning
of energy-minimization and Figure \ref{fig:jumppoisson:center} shows
the case of no preconditioning. Figure \ref{fig:jumppoisson:right} shows the case of diagonal preconditioning with
$K=10^3$. Comparing Figures \ref{fig:jumppoisson:left} and \ref{fig:jumppoisson:center},
we see that using preconditioning in weighted energy minimization reduces the number of
iterations necessary to achieve good convergence. Moreover,
preconditioning appears to actually improve
the best achievable AMG convergence factor in practice. For constrained energy-minimization,
energy minimization iterations without preconditioning increases the WPD by $3-5\times$ within a reasonable
number of iterations on $P$ (of course, asymptotically the preconditioned and non-preconditioned results are
equivalent but, in practice, only $O(1)$ iterations are done.) This raises an interesting question as to if better
preconditioners for energy minimization can actually improve the AMG solver's performance in a way that additional
iterations with a diagonal preconditioner cannot in practical time; however, this is a topic for future study.

Focusing on the more practical solvers in Figure \ref{fig:jumppoisson:left}, we also see that the results mirror those in Figure
\ref{fig:poisson}. Overall, the constrained energy-minimization case performs best, with weighted energy-minimization able to
approach the constrained case only for the right $\tau$ values and enough energy-minimization iterations on $P$. Again, there
is a limit beyond which additional energy-minimization iterations no longer improve AMG convergence. For constrained
energy-minimization, relatively few iterations are needed. Lastly, enforcing the constraint exactly or near-exactly is fundamental
to good convergence. Using energy-minimization with larger $\tau$ values leads to poor performance. 

The effects of the oscillating coefficient $K$ can be seen by comparing Figures \ref{fig:jumppoisson:left} and
\ref{fig:jumppoisson:right}. Interestingly, the larger $K$ value leads to a need for smaller $\tau$ values for the weighted case,
(compare the curves for $\tau=10^{-7}$). Overall, apart from changing the size of beneficial $\tau$ values, the size of the coefficient
oscillation does not noticeably affect either the weighted or constrained energy-minimization.

A final note of interest is that larger interpolation sparsity patterns do not help here. Thus, a moderate sparsity pattern of degree
three is chosen for these results.   

\begin{figure}[!ht]
  \centering
  \begin{subfigure}[b]{0.32\textwidth}
    \includegraphics[width=\textwidth]{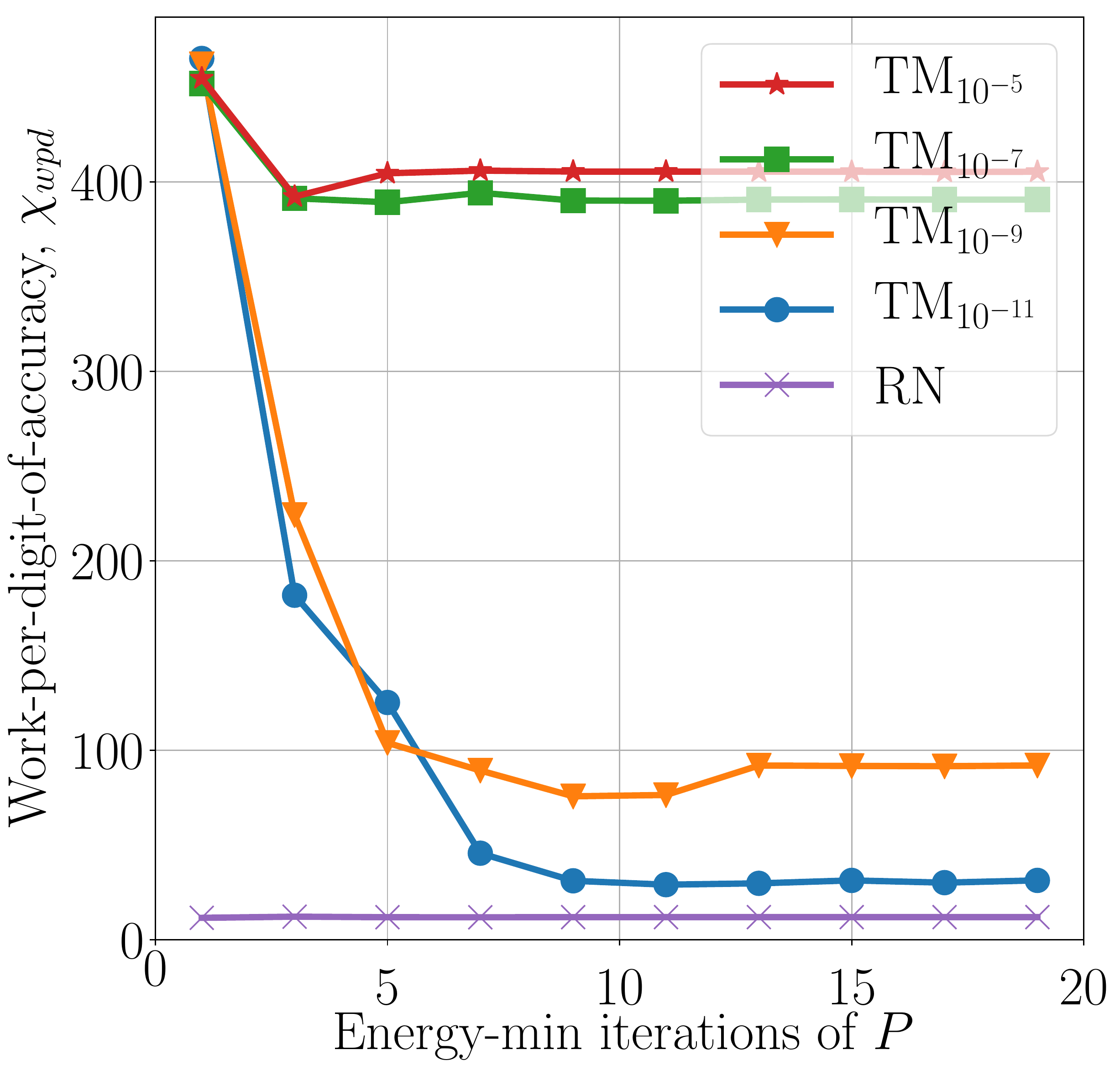}
    \caption{With diag. precon., $K = 10^6$}
    \label{fig:jumppoisson:left}
  \end{subfigure}
   \begin{subfigure}[b]{0.32\textwidth}
    \includegraphics[width=\textwidth]{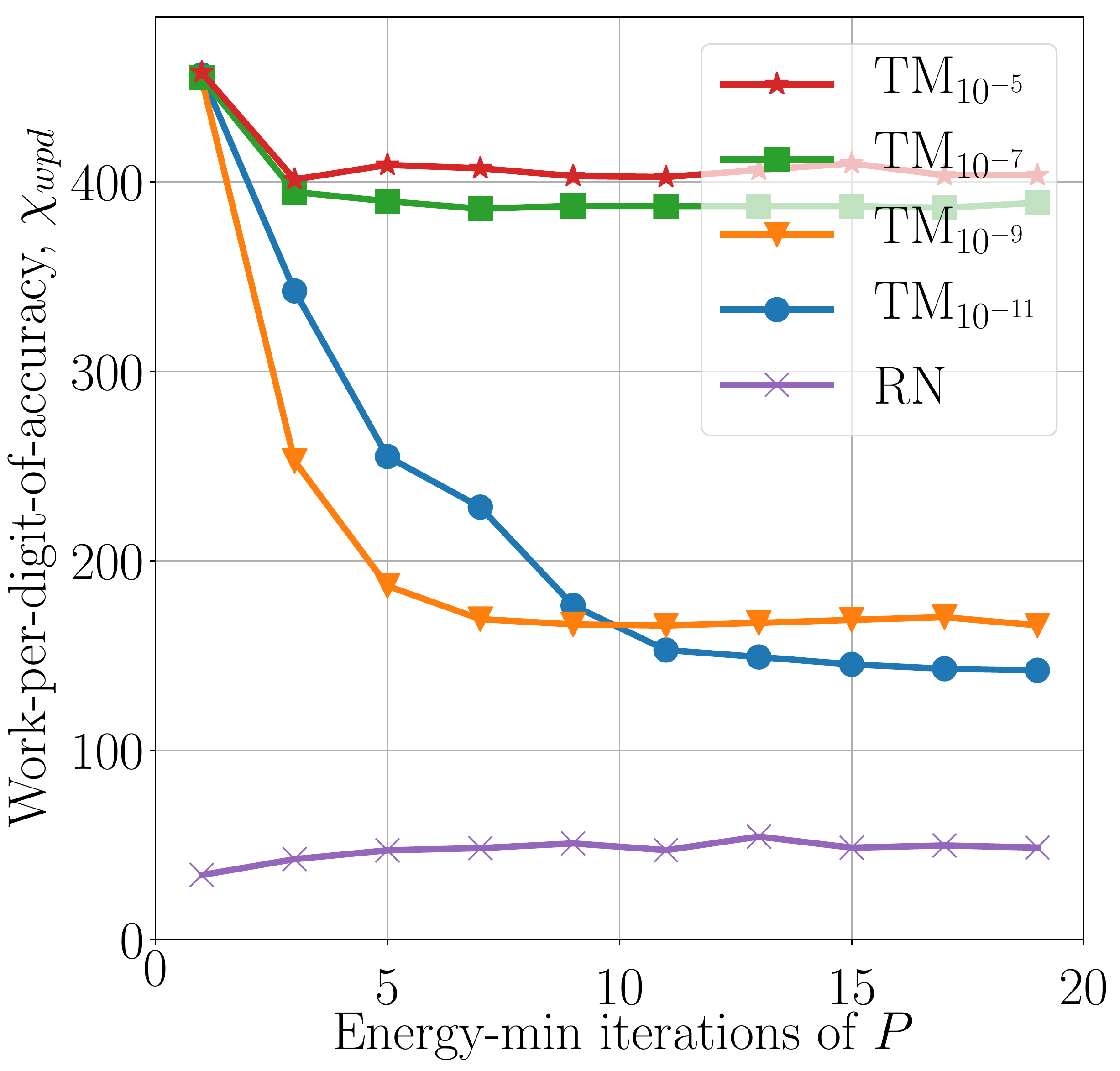}
    \caption{Without diag. precon., $K = 10^6$}
    \label{fig:jumppoisson:center}
  \end{subfigure}
  \begin{subfigure}[b]{0.32\textwidth}
    \includegraphics[width=\textwidth]{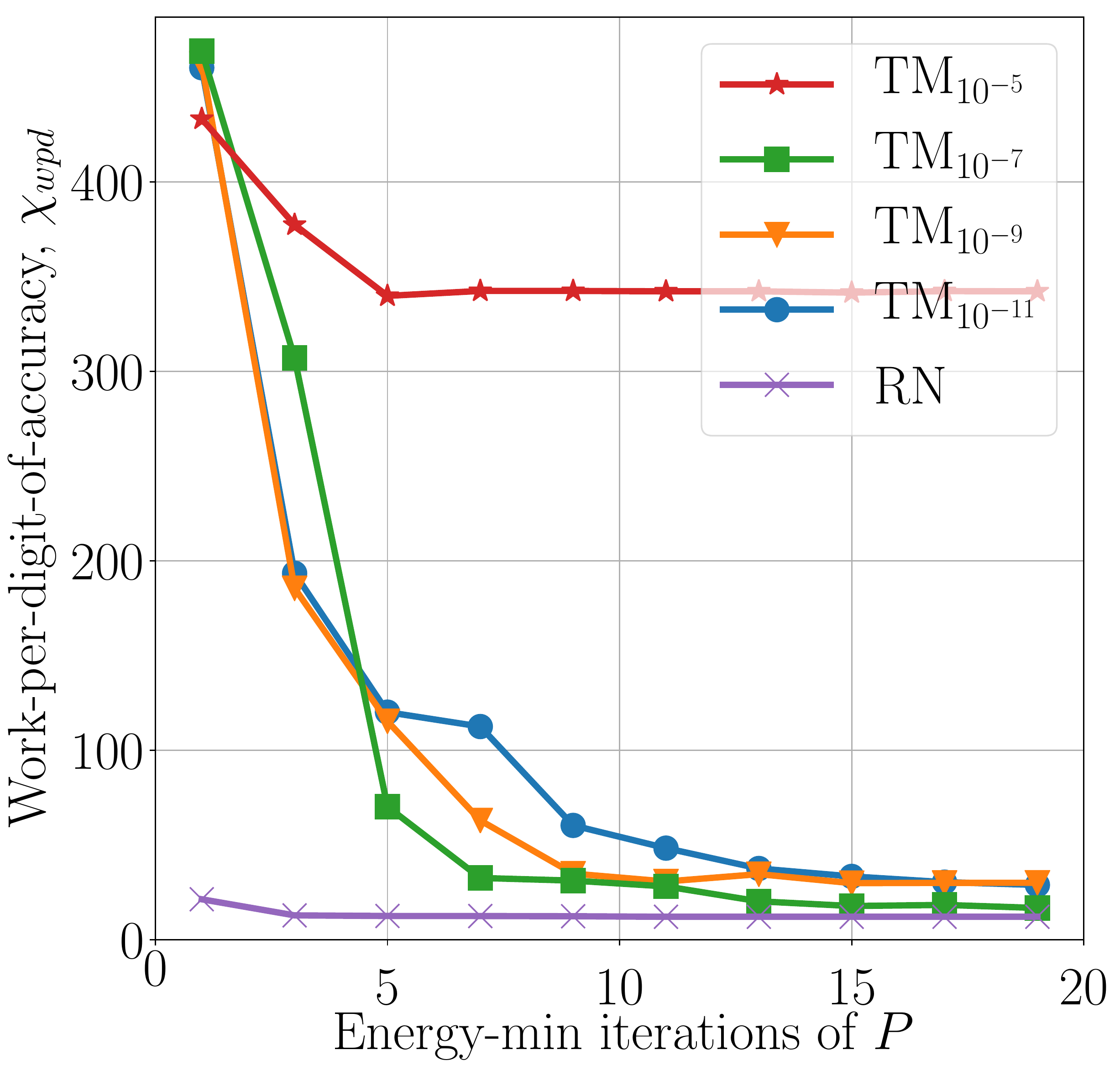}
    \caption{With diag. precon., $K = 10^3$}
    \label{fig:jumppoisson:right}
  \end{subfigure}
  \caption{Work-per-digit of accuracy, comparing weighted and constrained 
   energy-minimization, the use of diagonal preconditioning and two different coefficient jumps.}
  \label{fig:jumppoisson}
\end{figure}

\begin{remark} 
   We did not find tracking the CG residual norm during energy-minimization to
   be useful and, hence, omit plots of this information.  The key difficulty is
   that it is not clear how to connect the residual norm to the eventual
   multigrid convergence rate.  In other words, it is not clear how to use the
   residual norm to halt the energy-minimization process.  For instance, taking
   the cases of constrained energy-minimization from Figures \ref{fig:poisson} and
   \ref{fig:jumppoisson}, it is clear that at most five iterations of
   energy-minimization are needed.  However, the residual norm continues to
   decrease monotonically by multiple orders of magnitude from iteration five
   to iteration 19.  Yet, this extra residual reduction does not speed up
   convergence of the resulting multigrid solver.  In practice, the
   number of iterations needed typically equals the
   degree of the sparsity pattern of $P$ plus some small number, usually two or
   three.  This number of iterations is required to first fill the allowed sparsity
   pattern, and then to provide two or three iterations of additional improvement.
   \end{remark}

\subsection{Constraint vectors and adaptivity}\label{sec:results:target}

In Section \ref{sec:results:P}, we learned two things: (i) for good convergence, it is important that $P$ exactly or almost exactly
interpolates an appropriate constraint vector, and (ii) coupled with a good constraint, energy minimization can improve convergence,
but only by a fixed amount. This leads to the natural idea of adding an additional constraint vector when further energy minimization
of $P$ no longer improves convergence. Such an approach is the basis
of adaptive multigrid methods, where a set of constraint
vectors are developed that are then included or approximately included in the range of $P$ \cite{Brandt:2011hb,Brezina:2004eh,
DAmbra:2013iwa}. There are multiple ways
to generate constraint vectors; here we take the simple approach of generating a random vector $\mathbf{x}_0$ and applying some
form of improvement iterations (either relaxation or V-cycles) to reduce $\|\mathbf{x}_0\|_A$. Table \ref{tab:poisson} shows results
for constrained energy minimization AMG applied to the anisotropic
Poisson problem, with varying numbers of improvement iterations and
varying numbers of constraint vectors. 

\begin{figure}[!h]
  \centering
\begin{subfigure}[b]{0.475\textwidth}
  \begin{tabular}{c c c c c c c}     Vecs & Imp. Iters & OC & CC & CF \\\midrule
	 1 & 2 & 1.52 & 5.97 & 0.75 \\
	 2 & 2 & 1.55 & 6.01 & 0.74 \\
	 3 & 2 & 1.56 & 6.02 & 0.79 \\\midrule
	 1 & 5 & 1.51 & 5.95 & 0.64 \\
	 2 & 5 & 1.54 & 5.99 & 0.70 \\
	 3 & 5 & 1.55 & 6.01 & 0.73 \\\midrule
	 1 & 10 & 1.50 & 5.95 & 0.53 \\
	 2 & 10 & 1.54 & 5.98 & 0.67 \\
	 3 & 10 & 1.55 & 5.99 & 0.69 \\\midrule
	 1 & 25 & 1.50 & 5.95 & 0.49 \\
	 2 & 25 & 1.54 & 5.97 & 0.67 \\
	 3 & 25 & 1.55 & 5.98 & 0.65 \\\midrule
	 1 & 100 & 1.50 & 5.95 & 0.48 \\
	 2 & 100 & 1.50 & 5.95 & 0.50 \\
	 3 & 100 & 1.50 & 5.95 & 0.51 \\\midrule
  \end{tabular}
  \caption{Two-grid}
\end{subfigure}
\begin{subfigure}[b]{0.475\textwidth}
  \begin{tabular}{c c c c c c c}     Vecs & Imp. Iters & OC & CC & CF \\\midrule
	 1 & 2 & 1.64 & 9.39 & 0.76 \\
	 2 & 2 & 1.67 & 9.52 & 0.81 \\
	 3 & 2 & 1.67 & 9.50 & 0.85 \\\midrule
	 1 & 5 & 1.63 & 9.29 & 0.64 \\
	 2 & 5 & 1.66 & 9.44 & 0.78 \\
	 3 & 5 & 1.66 & 9.47 & 0.83 \\\midrule
	 1 & 10 & 1.62 & 9.27 & 0.54 \\
	 2 & 10 & 1.64 & 9.36 & 0.76 \\
	 3 & 10 & 1.66 & 9.42 & 0.82 \\\midrule
	 1 & 25 & 1.62 & 9.23 & 0.51 \\
	 2 & 25 & 1.66 & 9.32 & 0.69 \\
	 3 & 25 & 1.66 & 9.42 & 0.78 \\\midrule
	 1 & 100 & 1.62 & 9.28 & 0.50 \\
	 2 & 100 & 1.62 & 9.28 & 0.54 \\
	 3 & 100 & 1.62 & 9.27 & 0.54 \\\midrule
  \end{tabular}
   \caption{Multigrid}
\end{subfigure}
  \caption{Constrained energy minimization applied to a strongly
    anisotropic diffusion problem ($\epsilon=0.001$) in a two-grid
  and multigrid method. Constraints are initialized as a random vector; for the first constraint, Jacobi iterations are
  applied as improvement iterations. After an AMG hierarchy has been formed with one target, a new random vector
  is generated and V-cycles are applied as improvement iterations to generate a second target. The hierarchy is rebuilt
  using the new constraints, and so on. }\label{tab:poisson}
\end{figure}

Several interesting things follow from the results in Table \ref{tab:poisson}. First, the difference in convergence factor
between two-grid and multigrid is very small. This indicates that we are solving our coarse-grid problem well using V-cycles, and
that convergence is limited by how ``good'' the coarse-grid problem is, and not how accurately we are solving it. Moreover, naively adding
constraint vectors that were not accounted for in the range of $P$ does not improve convergence and, in fact, degrades
convergence in all cases, while increasing the setup complexity. Although more involved processes have been developed for
adaptive multigrid methods, these simple tests give insight that improving convergence is not as simple adding new constraint
vectors.

\section{Conclusions}\label{sec:conclusions}

This paper explores the role of energy minimization in AMG interpolation from a theoretical and practical
perspective. The eventual goal is to develop improved interpolation techniques that are more robust than current
state-of-the-art, without the significant overhead setup cost of fully adaptive methods. A minimization framework is
developed based on a weighted combination of interpolating known low-energy modes with a global energy
minimization over $P$. On one hand, accurately interpolating the constraint vectors proves to be of fundamental
importance to good convergence, as observed where constrained energy minimization consistently performs best,
and weighted energy minimization performs best with the relative weight of interpolating constraints $\gg 0.99$.
However, convergence generally does not improve when additional constraint vectors are added beyond the first.
This either means, for these test problems, (i) accurately interpolating one constraint vector leads to convergence
factors close to the optimal rate for the given coarse grid \cite{brannick2017optimal}, or (ii) there are other factors
fundamental to convergence of AMG that are not being addressed in this framework. Results here do not suggest
the newly proposed algorithm is superior to existing methods, but do provide insight on the convergence of
AMG in the practical setting, as well as the relation to AMG convergence theory. 

\section*{Acknowledgments}
The authors gratefully acknowledge the contributions of Ludmil
Zikatanov to the work presented here.

\bibliographystyle{siam}
\bibliography{tm_new.bib}

\end{document}